\documentclass{amsart}

\usepackage{amssymb}

\usepackage{xcolor}

\let\labelx\label
\def\label#1{\labelx{#1}\marginpar{#1}}

\DeclareMathSymbol\restr\mathbin{AMSa}{"16}  \let\restriction\restr
\DeclareMathSymbol\PP\mathord{AMSb}{`P}
\DeclareMathSymbol\le\mathrel{AMSa}{"36}
\DeclareMathSymbol\ge\mathrel{AMSa}{"3E} 
\DeclareMathSymbol\emptyset\mathord{AMSb}{"3F}

\newcommand\dom{\operatorname{dom}}

\newcommand\val{\operatorname{val}}

\newcommand\sakne{\mathsf{stem}}

\theoremstyle{plain}
 \newtheorem{theorem}{Theorem}
\newtheorem{proposition}[theorem]{Proposition}
\newtheorem{corollary}[theorem]{Corollary}
\newtheorem{lemma}[theorem]{Lemma}
\newtheorem{definition}[theorem]{Definition}
\theoremstyle{remark}

\newtheorem{claim}{Claim}

\usepackage{amsrefs}

\begin{document}

\title[An Efimov space with character less than $\mathfrak s$]
{An Efimov space with character less than $\mathfrak s$}

\author{Alan Dow}

\begin{abstract} It is consistent that there is a
compact space of
  character less than the splitting number in which
  there are no converging sequences. Such a space is an Efimov
  space.
\end{abstract}

\date{July 21, 2019}


\maketitle

An infinite compact space is an Efimov space if it contains no
converging sequences and no topological copies of $\beta\omega$. It
remains a central open problem to determine if such a space exists in
ZFC. The splitting number $\mathfrak s$, one of the basic 
and well-studied 
cardinal invariants of the continuum, is defined
as the minimum cardinality of a splitting family (defined below).
The splitting number has been shown to equal
 the minimum
weight of a compact space in which there is an infinite
sequence with no converging subsequence.
A related cardinal invariant, $\mathfrak z$, is defined
as the minimum weight of an infinite compact space
containing no converging sequences. This was introduced and studied
 by D. Sobota and also 
studied in \cite{withWill}. Evidently  $\mathfrak s\leq \mathfrak z$
and the weight of
every Efimov space is at least $\mathfrak z$. 
It was shown by
Hausdorff that  $\beta\omega$ has character
 $\mathfrak c$ and so a compact space  of character
 less than $\mathfrak s$ will not contain a copy of $\beta \omega$.
Other connections between the splitting number and Efimov's problem
have been established. It was shown in \cite{SplitCharacter} that
 barring
 large cardinals, if $\mathfrak s < \mathfrak c$, then there
 is an Efimov space and it was observed in \cite{withWill} that
under the same hypotheses an Efimov
space can be constructed to have character equal to $\mathfrak s$.
The space constructed in \cite{ShelahEfimov} was the first to be
constructed in a model of $\mathfrak s =\mathfrak c>\aleph_1$. 
We use the same basic method for the construction of the space.

Another line of inquiry,  concerning pseudointersection numbers of
ultrafilters on $\omega$, is related and 
led us to consider the possibility that
it may be possible for an Efimov space to have character less
than $\mathfrak s$.  
The pseudointersection number of  an ultrafilter
on $\omega$ is the minimum cardinality of a subset
of the filter for which there is no infinite set that is mod finite
contained in every member, i.e. a pseudointersection.
 Every infinite closed subspace of
an Efimov space is again Efimov and so can be assumed
to have a countable dense discrete subset. The trace of every
neighborhood filter on that countable dense is a filter
with no pseudointersection and it follows that every
free ultrafilter on $\omega$ contains such a filter.
Therefore, the example we construct in this paper produces
a model in which the pseudointersection is less
than the splitting number.   This was already established 
in \cite{BrendleShelah}.
 This was explored further in
\cite{ShelahSplitting} to show that the splitting number can be much
larger than the pseudointersection number of every ultrafilter.
A common feature of these papers is that in order to make
the value of $\mathfrak s$ large it is necessary to introduce
a pseudointersection for at least one filter
 while preserving that the many filters witnessing 
 small pseudointersection numbers are preserved to have
 no pseudointersection.

\section{Preliminaries and outline}

In this section we review the techniques and background used
throughout the paper. The set-theoretic aspects include
finite-support iterated forcing of $\sigma$-centered (hence
ccc) posets, the Laver version of a poset for adding a
pseudointersection to a filter, and preservation results about
not adding uncountable branches to trees. The topological
aspects include  Stone duality with Boolean algebras, 
recursive constructions of Ostaszewski style spaces,  the theory
of minimally generated Boolean algebras and 
the generalization to $\mathbb T$-algebras.

We begin with the promised definition of the splitting number.
A family $\mathcal S$ of subset of $\omega$ is a splitting
family if, for every infinite $a\subset \omega$, there is an
$S\in \mathcal S$ such that each of
 $a\cap S$ and $a\setminus S$ are infinite;
  we say that $S$ \textit{splits\/} the set $a$.  
 The splitting
 number $\mathfrak s$ is the minimum cardinality of a splitting
 family.
 The connection to pseudointersections is the following
 easily proven proposition.
 
 \begin{proposition} If $\mathcal B$ is a Boolean subalgebra
 of $\mathcal P(\omega)$ of cardinality less than $\mathfrak s$,
  then there is an ultrafilter of $\mathcal B$ that has an
  infinite pseudointersection.
  \end{proposition}

Readers familiar with Mathias forcing will know that with
 $\mathcal B = \mathcal P(\omega)$, the forcing
 will add a new free ultrafilter on $\omega$
 (i.e. $\mathcal B$) and
 a pseudointersection for it, while preserving that none
 of the ground model ultrafilters have a pseudointersection.
 We will be using the following variant in the style
 of
 Laver forcing  (\cite{Laver}) that
was introduced in 
 \cite{Blass1,JudahShelah}. This variant is $\sigma$-centered meaning
 that the poset can be written as a countable union of centered
 subsets.
 
 The poset is a set of subtrees of $\omega^{<\omega}$. A subtree
 $T_1$
 of a tree $T$ means that all predecessors in $T$
  of every element of $T_1$ is also in $T_1$ and that $T_1$ is
  ordered with the inherited order. 
  We refer to elements $t\in \omega^{<\omega}$ as
nodes and use the notation $t^\frown j$ for the node $t\cup \{ (\dom(t),j)\}$.

\begin{definition} For a filter $\mathcal D$ on $\omega$, the poset
  $\mathbb L(\mathcal D)$ denotes the set of sub-trees
   $T$ of $\omega^{<\omega}$ satisfying the following properties:
   \begin{enumerate}
   \item for each $t\in T$ and
    $s\subset t$ in $\omega^{<\omega}$, $s\in T$,
   \item a node $t\in T$ is said to be a branching node of $T$ if the
    set $\{ j\in \omega : t^\frown j\}$ has more than one element,
    \item the minimum branching node of $T$ is called the stem of $T$
    and is denoted with $\sakne(T)$,
    \item $T$ is 
  everywhere $\mathcal D$-branching above the stem in the sense
  that for all $\sakne(T)\subseteq t$, $\{ j \in\omega : t^\frown j\in T\}\in \mathcal D$.
  \end{enumerate}
  $\mathbb L(\mathcal D)$ is ordered by inclusion. We use the
  standard notation
  $T_1<_0 T_0$ to indicate that $T_1\subset T_0$ and  
  $\sakne(T_1)=\sakne(T_0)$.
  \end{definition}

 It is well-known
 that if $\mathcal D$ is a free filter
 then $\mathbb L(\mathcal D)$  adds a dominating real
 and if $\mathcal D$ is a free ultrafilter
 then $\mathbb L(\mathcal D)$ adds a subset of $\omega$
 that is not split by any set from the ground model
 (see \cite{Blass1}*{Theorem 8}). 
 We will assume without mention that ultrafilter on $\omega$
 refers to a free ultrafilter. 
\medskip

  For any  ordinal $\mu$, we use the following conventions:
a finite support iteration of $\sigma$-centered posets,
 $\langle  P_\alpha  ,\dot Q_\alpha :
  \alpha <\mu\rangle$ with limit $P_\mu$, 
means that  for all $\alpha\leq \mu$
\begin{enumerate}
\item $P_0$ is the trivial poset $\{\emptyset\}$;
\item $P_\alpha$ is the limit of the sequence 
$\langle  P_\beta,  \dot Q_\beta :
  \beta<\alpha  \rangle$;
  \item for $\alpha<\mu$, $\dot Q_\alpha$ is a $P_\alpha$-name of a $\sigma$-centered poset;
  \item for each $p\in P_\alpha$, $p$ is a function with finite domain, $\dom(p)$, contained
  in $\alpha$, and $p(\beta)\in \dot Q_\alpha$ for each $\beta\in \dom(p)$;
  \item for $p,q\in P_\alpha$, $p<q$ providing $\dom(p)\supset \dom(q)$, and
  for all $\beta\in\dom(q)$, $p\restriction\beta\Vdash p(\beta)\leq_{\dot Q_\beta}q(\beta)$.
\end{enumerate}
 
Another tree that will feature prominently in our construction is
the full binary tree of height $\omega_1$: $2^{<\omega_1}$. 
Utilizing
the theory of $\mathbb T$-algebras, our main construction will
be of a Boolean algebra with a generating family indexed by
the nodes of the tree $2^{<\omega_1}$ in such a way 
that the ultrafilters canonically correspond to maximal
(hence uncountable)  branches of $2^{<\omega_1}$.
 More importantly, the family
of generators indexed by the nodes of the corresponding branch 
will  generate a base for the ultrafilter. This next result is
critical for our goal of ensuring that no such filters acquire a pseudointersection.
It will be our convention in this paper to let $\Gamma$ refer to subtrees
of $2^{<\omega_1}$. In this next proposition, $\Gamma$ refers
to the members of $2^{<\omega_1}$ that are in the ground model.

\begin{proposition} \cite{KunenTall}
Let $\Gamma =2^{<\omega_1}$. 
In a forcing extension by a finite support iteration of
$\sigma$-centered posets, all
uncountable branches of $\Gamma$ belong to the
ground model.
\end{proposition}

\medskip

Now we turn to the topology background. 
An Ostaszewski style space will simply mean a locally countable and 
locally compact topology on $\omega_1$ in which the set $\omega$
is open and dense. In fact, we also intend that every initial segment
 $\alpha$ is an open subset. Such spaces are scattered and the topology can
 be specified by presenting a sequence $\langle U_\alpha : \alpha \in \omega_1\rangle$ where each $U_\alpha$ is a 
subset of the initial segment 
$\alpha{+}1$ and is a compact clopen neighborhood of 
the point $\alpha$.  Furthermore, since $\omega$ is dense, 
 we can set $a_\alpha = U_\alpha \cap \omega$ and consider
 the Stone space of the Boolean subalgebra of $\mathcal P(\omega)$
 generated by the family $\{ a_\alpha : \alpha \in \omega_1\}$. 
 We identify the point $\alpha$ with the ultrafilter generated by
 the family $\{ a_\alpha \}\cup \{\omega\setminus a_\beta : \beta < \alpha\}$,
and we have constructed the same space. The one-point compactification of
$\omega_1$ with this topology is the Stone space of this same Boolean
algebra. 

We also wish to connect this presentation of Ostaszewski style spaces
to Koppelberg's work on minimally generated Boolean algebras
(\cite{Koppelberg2}) where (minimal) generating sequences for
Boolean
algebras are introduced and discussed. 
An important connection to Efimov's problem proven
in \cite{Koppelberg2} is that the Stone space of a 
minimally generated Boolean algebra will not contain 
a copy of $\beta\mathbb N$ (in fact she proved that the Boolean
algebra does not contain an uncountable free algebra).
 
As in \cite{Cclosed}, this 
motivates the terminology in this next definition and the adjective
\textit{coherent} refers to the optional coherence in the generation of
the family of 
ultrafilters.  
The family of sets $\{ a_\alpha : \alpha\in\omega_1\}$ discussed in
the previous paragraph.
By abstracting the properties it makes it easier
 to verify the properties of our construction.

\begin{definition} A sequence $\langle a_\alpha
  :\alpha\in\lambda \rangle$ of subsets of $\omega$
 is  coherent if for each $\alpha\leq\beta<\lambda$ there is a finite 
 $F\subseteq
 \alpha$ such that either $a_\alpha\cap a_\beta\subseteq   
\bigcup_{\gamma\in
   F}a_\gamma = a_F$ or
 $a_\alpha \subseteq a_\beta\cup a_F$.  The sequence is
 proper if for all $\beta$, $a_\beta$ is not in the ideal generated
 by $\{a_\xi : \xi < \beta\}$. 
 \end{definition}
 
 In our motivating example, $a_\alpha$ is defined to be $U_\alpha\cap \omega$,
 and now we work in reverse and use $\widehat a_{\alpha}$ as notation
 for $U_\alpha$.
 
 \begin{definition}
 If $\langle a_\alpha : \alpha\in \lambda\rangle$ is a proper
 coherent sequence of subsets of $\omega$ and if $\beta < \lambda$,
 $\widehat a_{\beta}$ is defined
 to be the set of $\alpha\leq \beta $ such that $a_\alpha\setminus a_\beta$
is contained in $a_F$ for some finite $F\subseteq\alpha$.
 Observe that ${\widehat
   a}_\beta\subseteq \beta+1$. For each non-empty finite $F\subset\omega_1$,
 let $\widehat{a}_F = \bigcup_{\alpha\in F} \widehat{a}_\alpha$.  
\end{definition}

Note that every initial segment of a coherent sequence  is
also a coherent sequence, and that 
 the value of $\widehat{a}_\alpha$ depends only on the initial
 sequence up to $\alpha+1$.

\begin{proposition}[\cite{Koppelberg2}] If $\mathcal A = 
\langle a_\alpha : \alpha \in \lambda\rangle$ is a\label{two}
proper  coherent sequence, then the family $\{ {\widehat a}_\beta :
\beta \in \lambda\}$ generates a  compact scattered Hausdorff topology,
 $\tau(\mathcal A)$,
 on
$\lambda+1$ in which each ${\widehat a}_\beta$ is compact and open.
\end{proposition}

\begin{proof}  The family $\mathcal S = 
\{ \widehat{a}_\beta : \beta <\lambda\}
  \cup \{\lambda+1\setminus \widehat{a}_\beta : \beta < \lambda\}$
  forms
  a subbase for the topology. The topology is Hausdorf since
  $\alpha\in \widehat{a}_\alpha$ and $\beta \in \lambda+1\setminus
  \widehat{a}_\alpha$ for all $\alpha < \beta\leq \lambda$. 
  Similarly, the space is scattered since $\min(Y)$ is isolated in $Y$
  for all $Y\subset \lambda{+}1$.
We use the Alexander Subbase Theorem to prove that the space is
compact. Let $\mathcal B\subset \mathcal S$ be a cover of
$\lambda+1$. Let $H$ be the set of $\gamma\leq\lambda$
such that the interval $[\gamma,\lambda]$ is covered by a finite
subset of $\mathcal B$. Since $\lambda\in H$,
$H$ is not empty and we  let $\gamma_0$ be the minimum element
of $H$. Assume that $\gamma_0>0$. By the minimality of $\gamma_0$
it follows that $\lambda+1\setminus \widehat{a}_\beta\notin \mathcal
B$ for all $\beta <\gamma_0$ and so
there must be a $\beta<\lambda$ such that
$\gamma_0\in \widehat{a}_{\beta} \in \mathcal B$. Since the sequence
is coherent, there is a finite $F\subset \gamma_0$ satisfying
that $\widehat{a}_{\gamma_0} \subset \widehat{a}_\beta\cup
\widehat{a}_F$. But now it follows easily
that $\max{F}\in H$, contradicting the minimality of $\gamma_0$.
\end{proof}

Finally, we introduce the remarkable generalization
of linear coherence to  a coherence structure based on
 binary trees. These are the $\mathbb T$-algebras
 introduced
in \cite{Kosz1}. However for this paper we restrict to subalgebras of
 $\mathcal P(\omega)$ and subtrees of $2^{<\omega_1}$.
 Also, our notation will
more closely follow those of
 \cite{Roberto,ShelahEfimov} but especially \cite{Cclosed}.  
For a node $\sigma\in 2^{<\lambda}$ for any ordinal $\lambda$,
 define $\sigma^\dagger$ to be $\sigma$ if $\dom(\sigma)$ is not a
 successor ordinal, and otherwise $\mathop{dom}(\sigma^\dagger)
  = \dom(\sigma)$ and $\sigma^\dagger$ agrees with $\sigma$
  except at the maximum value in $\mathop{dom}(\sigma)$. For
  each $x\in 2^{\leq\omega_1}$, let $\lambda_x$ denote the domain of
  $x$.

\begin{definition} A $\mathbb T$-algebra is a family\label{fifteen}
  $\mathcal A_{\Gamma} =
\langle a_\sigma : \sigma \in \Gamma \rangle$ such that,
   for some ordinal $\lambda$
  
  \begin{enumerate}
  \item $\Gamma \subset 2^{<\lambda}$ is a subtree of $2^{<\lambda}$ that is 
  downward closed (i.e. $\tau\in \Gamma$ for every $\sigma\in \Gamma$
  and $\tau\subset \sigma$ in  $2^{<\lambda}$),
  \item if $\sigma\in \Gamma$, then $\sigma^\dagger\in \Gamma$
     (say that $\Gamma$ is twinned),
  \item $a_\sigma$ is empty for each $\sigma\in \Gamma$
    that is not on a successor level,
  \item for $\sigma$ on a successor level of $\Gamma$,
    $a_{\sigma^\dagger} = \omega\setminus a_\sigma$,
  \item for each  $x\in 2^{\leq\lambda}$,
    the family $\mathcal A_{\Gamma,x} = 
    \{ a^x_\alpha = a_{x\restriction\alpha{+}1} :
       x\restriction \alpha{+}1 \in \Gamma\}$ is a proper coherent
       sequence.  
    \end{enumerate}
\end{definition}

Henceforth $\Gamma$ will always denote a subtree of
 $2^{<\omega_1}$ that is downward closed, 
 has no maximal elements, and is 
closed under the $\dagger$ operation. For each such $\Gamma$, 
 $X(\Gamma)$ will denote the set of minimal elements
 of $2^{\leq\omega_1}\setminus \Gamma$ (i.e. the maximal branches of
 $\Gamma$). We will let $X(\Gamma,\aleph_0)$ denote
 $\{ x\in X(\Gamma) : \lambda_x<\omega_1\}$ (the countable elements of
 $X(\Gamma)$).

\begin{lemma}\cite{Kosz1} Let $\mathcal A_\Gamma = \langle a_\sigma : \sigma\in\Gamma\rangle$
be a $\mathbb T$-algebra. Let $x\in X(\Gamma)$ and let $\mathcal U$ be any 
ultrafilter on the Boolean subalgebra, $\mathcal B_\Gamma$, of $\mathcal P(\omega)$ generated by
 $\mathcal A_\Gamma$.
 \begin{enumerate}
 \item The family of finite intersections from the set
  $$\{ \omega\setminus a_{x\restriction \alpha} : \alpha\in\dom(x)\}
  \ \ \ \mbox{i.e.}\ \{ \omega\setminus a^x_\alpha : a^x_\alpha\in \mathcal A_{\Gamma,x}\}$$
 is a filter base for an ultrafilter $\mathcal U_x$ on $ \mathcal B_{\Gamma}$.
 \item  There is a $y\in X(\Gamma)$ such that $\mathcal U = \mathcal U_y$.
 \end{enumerate}
 \end{lemma}

\begin{proof} We first prove item (1). Since the family $\mathcal A_{\Gamma,x}$ 
is assumed to be a proper coherent sequence, it follows that no finite union from
 $\mathcal A_{\Gamma,x}$ is cofinite. To show that the filter $\mathcal U_x$ with
 the family from (1) as a base is an ultrafilter, it suffices to show that for
 every $\sigma\in \Gamma$, one of $a_\sigma$ or $a_{\sigma^\dagger}$ is
 in $\mathcal U_x$. This is immediate for all $\sigma\subset x$ so we may
 suppose there is an
  $\alpha\in \dom(x)$ such that $(x\restriction \alpha)\not\subset \sigma$
 and $(x\restriction \alpha)^\dagger\subseteq \sigma$.  If $\sigma=(x\restriction \alpha)^\dagger$,
 then $a_\sigma=\omega\setminus a_{x\restriction\alpha}$ is an element of $\mathcal
 U_x$. Otherwise, 
using that each of $\mathcal A_{\sigma{}^\frown 0}$ and
$\mathcal A_{\sigma^\dagger {}^\frown 0}$ are 
 coherent,  and that $a_{\sigma^\dagger}=\omega\setminus a_\sigma$,
 there is a  finite $F\subset \alpha$ such that one of 
 $\{a_{(x\restriction\alpha)^\dagger}\setminus a_{\sigma}\,\ \  
 a_{(x\restriction\alpha)^\dagger}\setminus a_{\sigma^\dagger}\}$
 is contained in $\bigcup \{ a_{x\restriction\xi{+}1} : \xi{+}1\in F \}$.
 By symmetry, assume that $a_{(x\restriction \alpha)^\dagger}\subset
  a_\sigma\cup \bigcup\{a_{x\restriction\xi{+}1} : \xi{+}1\in F\}$. In other words,
    $a_\sigma$ contains $(\omega\setminus a_{x\restriction\alpha})\cap
     \bigcap \{\omega\setminus a_{x\restriction\xi{+}1} : \xi{+}1\in F\}$,
      which is a member of $\mathcal U_x$.
      
      Now we prove (2). We define $y\in X(\Gamma)$ by recursively 
      defining  an increasing chain  in $\Gamma$.   
      Let $y\restriction 1=\sigma_1$ be chosen so that $a_{\sigma_1}\notin \mathcal U$.
      Assume that an
      increasing chain $\{ \sigma_\beta  \in 2^\beta : \beta <\alpha\}\subset \Gamma$ has
      been chosen so that $a_{\sigma_\beta}\notin \mathcal U$ for all $\beta <\alpha$.
      If $\alpha$ is a limit ordinal, then $\sigma_\alpha = \bigcup \{\sigma_\beta :
      \beta<\alpha\}$. If $\sigma_\alpha\notin \Gamma$, then 
      set $y=\sigma_\alpha\in X(\Gamma)$. By part (1), it follows that $\mathcal U_y = \mathcal U$.
            Assume that $\alpha=\beta+1$ and
      again, simply choose 
       $\sigma_\alpha\supset \sigma_\beta$ so that $a_{\sigma_\alpha}\notin \mathcal U$.
\end{proof}

One of the features of a $\mathbb T$-algebra is that the Stone space can be analyzed
by using the subalgebras (i.e. the Ostaszewski style spaces)
generated by the generators indexed by a given branch.  The space associated
with $\mathcal A_{\Gamma,x}$, for $x\in X(\Gamma)$, was described
in Proposition \ref{two}.

\begin{lemma} Let 
$\mathcal A_\Gamma$ be $\mathbb T$-algebra\label{sixteen} and let $x\in X(\Gamma)$. 
  the mapping
 $\varphi_x :X(\Gamma) \mapsto (\lambda_x+1,\tau(\mathcal
 A_{\Gamma,x}))$ is continuous where
 $\varphi_x(x)=\lambda_x$ and
 for all $x\neq y\in X(\Gamma)$,
 $\varphi_x(y)$ is the unique $\alpha<\lambda_x$ such
 that $(y\restriction\alpha+1)^\dagger = x\restriction\alpha+1$. In
 particular, $\varphi_x(x)\neq \varphi_x(y)$ for all $y\in X(\Gamma)\setminus \{x\}$.
\end{lemma}

\begin{proposition} If $\mathcal A_\Gamma$ is a $\mathbb T$-algebra with the
  ScP then any converging sequence ($\omega$-sequence)
  in $(X(\Gamma), \tau(\mathcal A_\Gamma))$\label{18}
converges to a
point in $X(\Gamma,\aleph_0)$. Every $x\in X(\Gamma,\aleph_0)$ 
 has countable character.
\end{proposition}

\section{The stationary covering property of coherent sequences}

However the stationary covering property was discovered in 
\cite{Roberto} as a tool for preserving countable compactness 
in what might usefully and informally  be called Ostaszewski style spaces. 
 A connection to the similarly important
 Scarborough-Stone problem was pointed out in \cite{ShelahEfimov} and
in the last section of this paper we briefly explore the 
connections to the splitting number in that setting.

This next property was introduced in
\cite{Roberto} under the name SP  
for the same purpose.

\begin{definition} Say that a coherent sequence
  $\mathcal A = \{ a_\alpha : \alpha\in \lambda\}$ 
 has the
 ScP (stationary covering property)
  if either, $\lambda\in\omega_1$ or
 for each stationary $S\subset \omega_1$,
  there is a finite $F\subset\omega_1$ such that
   $\{{\widehat a}_F\} \cup \{ {\widehat a}_\gamma : \gamma\in S\}$ is a
  cover of $\omega_1$.  Say that $S\cup F$ is a cover and
  that $S$ is an almost cover.
\end{definition}

We use the standard notation $\mathop{NS}_{\omega_1}$ to denote the
ideal of non-stationary subsets of $\omega_1$
and $\mathop{NS}_{\omega_1}^+$ denotes the set of stationary subsets
of $\omega_1$.

\begin{lemma}  If a proper coherent $\omega_1$-sequence\label{needScP}
  $\mathcal A$
  has the ScP, then $\mathcal A$ 
 induces a countably compact topology on $\omega_1$.
\end{lemma}

\begin{proof} Let $X$ be an infinite countable subset of $\omega_1$
  and suppose that $\widehat{a}_F\cap X$ is finite for all
  non-empty  finite $F\subset\omega_1$.   For each
  $\gamma >\sup(X)$, let $X_\gamma $ be the finite set
  $X\cap \widehat{a}_\gamma$. By the pressing down lemma, we
  may choose a stationary $S\subset \omega_1$ and a single
  finite set $Y\subset X$ so that $Y = X_\gamma$ for all $\gamma\in
  S$. We now have a contradiction since it follows that
   $S$ is not an almost cover. 
\end{proof}

For the remainder of the section assume that
 $\mathcal A = \{ a_\alpha : \alpha \in \omega_1 \}$ is  a coherent
sequence with ScP.

\begin{lemma} If $S\subset \omega_1$   is a\label{subset}
 cover,
   then for each stationary $S'$, there is a finite $F\subset
   S$ such that $S'\cup F$ is a cover.
\end{lemma}

\begin{proof} Choose finite $F_1$ such that $S'\cup F_1$ is a cover.
  For each $\alpha\in F_1$,  ${\widehat a}_\alpha$
is a compact space and $\{ {\widehat a}_\gamma : \gamma\in S\}$ is an
open cover. Therefore, there is a finite $F\subset S$ such that
$\bigcup_{\alpha\in F_1} {\widehat a}_\alpha $ is contained in
${\widehat a}_{F}$. This implies that $S'\cup F$ is a cover.
\end{proof}

\begin{lemma} If $S\subset\omega_1$ is a cover,
 then there is a finite\label{bigsets}
   $F\subset\omega_1$ such that for all $\alpha \in\omega_1\setminus
  {\widehat a}_F$, the set of $\gamma\in S$ such that
   $\alpha\in {\widehat a}_\gamma$ is uncountable. 
\end{lemma}

\begin{proof}  For each limit ordinal $\gamma\in\omega_1$,
  choose the minimal $\alpha_\gamma\in S$ such that
  $\gamma\in \widehat{a}_{\alpha_\gamma}$. For each limit
$\gamma$,
choose finite $H_\gamma\subset\gamma$ so that
$a_{\gamma}\setminus a_{\alpha_\gamma}\subset \widehat{a}_{H_\gamma}$.
This can be restated as $a_\gamma\subset a_{\alpha_\gamma}\cup
 \widehat{a}_{H_\gamma}$. 
By the pressing down lemma, we may
choose a stationary $S\subset\omega_1$ so that there is a finite
$H\subset\omega_1$ so that $H_\gamma=H$ for all $\gamma\in S$.
By Lemma \ref{subset}, choose, 
for each $\gamma\in S$,
 $F_\gamma\subset S$ so that
 $F_\gamma\cup (S\setminus\gamma)$ is a cover. By another application
 of the  pressing down lemma, there is a stationary $S'\subset S$
 and a finite $F$ so that $F=F_\gamma\cap \gamma$ for all $\gamma\in
 S'$. 

We show that $F\cup H$ is as required.
Fix any $\alpha \in \omega_1\setminus \widehat{a}_F$ and choose
any $\alpha<\gamma\in S'$.  We prove there is a $\delta\in
S\setminus\gamma$ with $\alpha\in \widehat{a}_\gamma$. 
Since $F \cup [(S'\cup S)\setminus\gamma]$
is a cover, $\alpha\in \widehat{a}_\beta$ for some
 $\beta\in (S'\cup (F_\gamma\setminus\gamma))\subset
 (S'\cup S)\setminus \gamma$.
If $\beta\in S$, then let $\delta=\beta$. 
 If $\beta\in S'\setminus S$,
  then $\alpha\in \widehat{a}_{\alpha_\beta}\cup
  \widehat{a}_{H}$. Since
  $\alpha\notin \widehat{a}_H$, then $\delta=\alpha_\beta$
  and we indeed have
  that $\alpha\in \widehat{a}_\delta$ for some $\delta\in S\setminus
    \gamma$. 
\end{proof}

This next result is our first mention of forcing. We refer the reader
to \cite{Kunen} for standard background and notation.
We will use the notation of placing a dot over a letter, such
as $\dot Q$, to indicate that we are refering to a name. 
Also, following \cite{Kunen}, we use
 $\check{x}$ as the notation for the canonical name of the set $x$
from the ground model (with the poset being clear from the context).
 For a poset
$P$ and a set $X$ we use the term nice $P$-name for a subset of $X$ to
mean a name of the form $\dot Y = 
\bigcup \{ \{\check{x}\}\times A_x : x\in
X\}$ where, for each $x\in X$, $A_x$ is a (possibly empty) antichain
of $P$. If $G$ is a $P$-generic filter, then $\val_G(\dot Y)$ denotes
the valuation or interpretation of the name $\dot Y$ and is defined
to be the set $\{ x\in X : G\cap A_x\neq\emptyset\}$. 
Every subset of $X$ in the forcing extension does equal the valuation
of a nice name, so it is often sufficient to only consider nice
names.

\begin{lemma}[\cite{Roberto}]
For a limit ordinal $\lambda$, 
if $\langle P_\xi , \dot Q_\xi : \xi \in \lambda\rangle$
  is a finite support\label{lemma7}
  ccc iteration such that, for all $\xi<\lambda$, $P_\xi$ forces
  that $\mathcal A$ has the ScP, then $P_\lambda$ forces that
  $\mathcal A$ has the ScP (hence Cohen forcing preserves that
  $\mathcal A$ has the ScP).
\end{lemma}

\begin{proof} Let $\dot S$ be a $P_\lambda$-name and let $p_0\in
  P_\lambda$ force that $\dot S\in \mathop{NS}_{\omega_1}^+$.
  We prove that $p_0$ does not force that $\dot S$ is a witness
  to the failure of the ScP.
Let $S_1$ denote the set of limit ordinals
$\xi\in\omega_1$ such that there is a $p_\xi<p_0$ with $p_\xi\Vdash
\xi\in \dot S$. Since $p_0\Vdash \dot S\subset S_1$ it follows that
$S_1$ is stationary. 
For each $\xi\in S_1$, let $H_\xi$ denote the
support of $p_\xi$ in the iteration sequence. By the standard
$\Delta$-system arguments, there is a finite $H$ and a stationary set
$S_2\subset S_1$ such that $H_\xi\cap H_\eta = H$ for all $\xi<\eta$
in
$S_2$. 

Fix any $\mu<\lambda$
such that $H\subset \mu$ and $p_0\in P_{\mu}$.
Let
 $\dot S_3 $ be the  nice $P_\mu$-name $\{
(\check{\xi},p_\xi\restriction\mu) : \xi\in S_2\}$.
We prove the well-known
 fact that there is
 $P_\mu$-generic filter $G_\mu$ satisfying that
 $\val_{G_\mu}(\dot S_3)$  is stationary. 
Since $P_\mu$ is ccc, it is a standard fact (see \cite[VII H1]{Kunen})
 that if $\dot C$ is a
$P_\mu$-name of a cub subset of $\omega_1$, then there is cub $C_1$
such that $1\Vdash C_1\subset \dot C$.   
Let $\dot C$ be a $P_\mu$-name of a cub
satisfying that $1$ forces that if $\dot S_3\in
\mathop{NS}_{\omega_1}$,
then $\dot S_3\cap \dot C$ is empty. Choose a cub $C_1$ as above,
 and choose any $\xi\in C_1\cap S_1$.  Since
 $p_\xi\restriction\mu\Vdash \xi \in \dot C\cap \dot S_3$,
 it follows that $p_\xi\restriction \mu$ forces that $\dot S_3$ is
 stationary. 

 Now let $G_\mu$ be a $P_\mu$-generic filter with $p_0\in G_\mu$
 and satisfying that
$S_4 = \val_{G_\mu}(\dot S_3)$ is
 stationary. We finish the proof by working in $V[G_\mu]$
 (where $\mathcal A$ has the ScP). 
 Choose any finite $F_0\subset \omega_1$ so that
  $S_4\cup F_0$ is a cover, and then choose finite $F_0\subset
 F\subset \omega_1 $ as in Lemma \ref{bigsets}. Thus, for each
 $\alpha\in\omega_1\setminus\widehat{a}_F$, the set of $\gamma\in S_4$
 with $\alpha\in \widehat{a}_\gamma$ is uncountable. We prove
 that if $q\in P_\lambda$ and $q\restriction \mu\in G_\mu$, then
 $q$ forces that $F\cup \dot S$ is a cover. Let $q\in P_\lambda$ with
  $q\restriction\mu\in G_\mu$ and fix any $\alpha\in
 \omega_1\setminus\widehat{a}_F$.   
Choose any $\xi\in S_4$ satisfying
 that $H_\xi\setminus H$ is disjoint from the support of $q$
 and that $\alpha\in \widehat{a}_\xi$. Since $p_\xi\restriction\mu$
 and $q\restriction \mu$ are in $G_\mu$, it follows that $p_\xi$
 and $q$ are compatible. Since $p_\xi\Vdash \xi\in\dot S$,
 it follows that $q$ does not force that $\alpha$ is not covered
 by $F\cup \dot S$. 
\end{proof}

 In this next definition we keep the notation simpler by assuming
  that $\mathcal D$ will be clear from the context.
  
  \begin{definition}  Let $\dot S$ be an $\mathbb L(\mathcal D)$-name
  for   a filter $\mathcal D$ on $\omega$. 
For $t\in \omega^{<\omega}$, say that $t\Vdash_w \gamma\in \dot S$ if
there is some $T\in \mathbb L(\mathcal D)$ such that
$T\Vdash \gamma\in \dot S$ and $t$ is the stem of $T$.
Also, for a name $\dot S$ and $t\in T$, let $S_t = \{\gamma\in
\omega_1 : t\Vdash_w \gamma\in \dot S\}$.
\end{definition}

\begin{lemma} 
  If $\mathcal D$ is a filter with countable character,
  then $\mathbb L(\mathcal D)$ preserves that $\mathcal A$
  has the  ScP.
\end{lemma}

\begin{proof}
  Let $\dot S$ be a name of a stationary set. For each
   $t\in \omega^{<\omega}$, let $S_t = \{ \gamma : t\Vdash_w \gamma\in
  \dot S\}$. We must prove that there is a $T\in\mathbb L(\mathcal D)$
 and an $F\in
       [\omega_1]^{<\aleph_0}$ such that $T\Vdash \dot S\cup F$ is a cover.
Note  that if $T\Vdash \dot S\notin
\mathop{NS}_{\omega_1}$, then $S_t\notin \mathop{NS}_{\omega_1}$ for all
$t\in T$.
Fix any
  $t\in \omega^{<\omega}$ such that
$S_{t}$ is stationary  Fix any finite $F\subset \omega_1$ such that $S_{t}\cup
F$ is a cover. Let $\{ D_n : n\in \omega\}$ be a descending base for
$\mathcal D$. For each $\gamma \in S_t$, choose $T_\gamma <_0
(\omega^{<\omega})_t$ forcing that $\gamma\in \dot S$.
For each $\gamma\in S_t$, choose $n_\gamma\in\omega$ so that
$D_{n_\gamma} \supset \{ j : t^\frown j\in T_\gamma\}$. Choose $n$
so that $S_t^n = \{ \gamma \in S_t: n_\gamma \leq n\}$ is stationary.
Note that $S_{t}^n\subset S_{t^\frown j}$ for all $j\in D_n$. 
By Lemma \ref{subset}, we may choose
a finite $F'\subset S_t$ so that $F\cup S_t^n\cup F'$ is a cover.
Since $F'\subset S_t$ we can increase $n$ and arrange that
 $F'\subset S_t^n$. It now follows that $S_{t^\frown j}\cup F$ is a
cover
for all $j\in D_n$. Continuing this construction we can produce
 $T_1<_0 (\omega^{<\omega})_t$ such that, for 
all $t\in T_1$, 
$S_t \cup F$ is a cover. It then follows easily that
 $\dot S\cup F$ is a cover. 
\end{proof}

\begin{lemma}
  Assume that $\mathcal D$ is an ultrafilter and that $\dot S$ is an
  $\mathbb L(\mathcal D)$-name of a stationary subset of $\omega_1$.
If, for\label{isacover}
 some $T\in \mathbb L(\mathcal D)$, $S_t$ is a cover for all
$t\in T$, then $T\Vdash \dot S$ is a cover.
\end{lemma}

\begin{proof}
  Fix any $\alpha $ and $T'<T$. Let $t$ be the stem of $T'$.
  Choose any $\gamma \in S_t$ so that $\alpha\in {\widehat a}_\gamma$
  and choose $T_\gamma <_0 T$ so that $T_\gamma \Vdash \gamma\in \dot
  S$. Then $T'\cap T_\gamma $ forces that $\gamma \in \dot S$. 
\end{proof}

\begin{proposition} For an $\mathbb L(\mathcal D)$-name $\dot S$,
if $S_t\in\mathop{NS}_{\omega_1}^+$ for all $t\in T\in \mathbb
L(\mathcal D)$, then $T\Vdash \dot S \in \mathop{NS}_{\omega_1}^+$.
\end{proposition}

\begin{proof}
Since $\mathbb L(\mathcal D)$ is ccc ($\sigma$-centered) it suffices
show that if $C$ is a cub and $T' < T$, then $T'\not\Vdash \dot S\cap
C$
is not empty. Of course this is immediate from the fact
that $S_t\cap C$ is not empty where $t=\sakne(T')$.
\end{proof}

\begin{lemma} If $\mathcal D$ is a filter and\label{allstat}
$T\Vdash \dot S\in \mathop{NS}_{\omega_1}^+$
  then there is $T'<_0 T$ such that $S_t\in \mathop{NS}_{\omega_1}^+$
  for all $t\in T'$.
\end{lemma}

\begin{proof}
Fix any $t\in T$ such that $S_t$ is stationary. Fix any countable
elementary submodel $M$ such that $T, \dot S$ are in $M$ and such
that $M\cap \omega_1 = \delta \in S_t$. Choose $T' <_0 T_t$
such that $T'  \Vdash \delta\in \dot S$. Evidently, 
$\delta\in S_{t'}$ for all $t'\in T' $. Since $\delta\in 
S_{t'}\in M$
for all $t'\in T'$, it follows that each such
$S_{t'}$ is in
$\mathop{NS}_{\omega_1}^+$. 
\end{proof}

\begin{lemma} If $\mathbb L(\mathcal D)$ forces that\label{thirteen} 
  $\mathcal A$ does not have the ScP,
  then there is a name $\dot S$ of a stationary set
  and a $T\in \mathbb L(\mathcal D)$ and $\{\{\beta_t\}\cup
F_t : t\in T\}
\subset [\omega_1]^{<\aleph_0}$ such that
  \begin{enumerate}
  \item  $S_t$  is stationary for all $t\in T$,
  \item $S_{t}\cup F_{t}\cup F_{\sakne(T)}$ is a cover, 
    \item for each $\sakne(T)\neq t\in T$,
        $\beta_{t}\in {\widehat a}_{F_{t}}\setminus {\widehat
            a}_{F_{{\sakne(T)}}}$, 
  \item for each $t,t^\frown j\in T$, $F_{t^\frown j}\subset S_{t}$,

  \item for each $T'< T$, there is a $t\in T'$ 
   such that the set 
     $$\bigcup_{t^\frown j\in T'} \{\beta_{t^\frown j}\}\setminus  
\left({\widehat a}_{F_{t}}\cup \bigcup \{ {\widehat a}_\gamma :
    \gamma\in S_{t^\frown j}\}\right)$$ is infinite.
  \end{enumerate}
\end{lemma}

\begin{proof} Fix $T$ forcing that $\dot S$ is stationary and
  that $\dot S$ is not an almost cover. Let $M$ be a countable
  elementary submodel so that there is a $t_0\in T$ with
 $M\cap \omega_1  \in S_{t_0}$. 
We may assume that $t_0 = \sakne(T)$  and, as in Lemma \ref{allstat},
 that $S_t$ is stationary for all $t\in T$. In fact, 
   we can assume that 
$\bigcap \{ S_{t\restriction \ell } : |t_0|\leq   \ell\leq  |t|\}$ 
is stationary for all $t\in T$.
Let $S_{t_0}^- = S_{t_0}$ and,   for each $t_0\subset t\in T$, let
   $S_t^-$ denote the stationary set 
\(\bigcap \{ S_{t\restriction \ell } : |t_0|\leq    \ell \leq
|t|\}\). We show that we may assume that $T_t\Vdash
 \dot S\subset S^-_t$ for all $t\in T$.
Consider the name $\dot S_1$ where
$$\dot S_1 = \{ (\gamma, \tilde T):
\tilde T \leq T,\  \tilde T\Vdash \gamma\in \dot S,
 \gamma \in S_{\sakne(\tilde T)}^-\}~.$$
 It is easily checked that, for all $t\in T$,
 $t\Vdash_w \gamma\in \dot S_1$ if and only if
 $\gamma \in S_t^-$.  Of course it follows, by Lemma \ref{allstat},
 that
  $T$ forces that $\dot S_1$ is a stationary 
subset of $ \dot S$ (and so is not an almost
  cover), that $\dot S_1\in M$, and that if we replace $\dot S$ by
  $\dot S_1$, then we have that $S_t^- = S_t$ for all $t\in T$.

Choose any finite 
$F_{t_0}$ so that $S_{t_0}\cup F_{t_0}$ is  a cover. 
By Lemma \ref{isacover} and by possibly replacing
 $t_0$ by an extension in $T$, 
 we may assume that $F_{t_0}\cup S_t$
is not a cover for all $t_0\subsetneq t\in T$.
 For all
$t_0\subsetneq t^\frown j\in T$,
choose a   finite $F_{t^\frown j}\subset S_t$
 so that $ S_{t^\frown j} \cup F_{t^\frown j} \cup F_{t_0}$ 
 is a   cover (by Lemma \ref{subset}).
 We now choose $\beta_{t^\frown j}$ using the fact
 that $S_{t^\frown j}\cup F_{t_0}$ is not a cover.
 If $S_{t^\frown j}\cup F_t\cup F_{t_0}$ is a cover,
 then choose $\beta_{t^\frown j}\in {\widehat a}_{F_t}
 \setminus
   {\widehat a}_{F_{t_0}}$. 
Otherwise,
let $\beta_{t^\frown j}$ be the minimal element
of ${\widehat a}_{F_{t^\frown j}}
\setminus \bigcup \{ {\widehat a}_\gamma : \gamma \in S_{t^\frown
  j}\cup F_t\cup F_{t_0}\}$. Note that $\beta_{t^\frown j}\notin {\widehat
  a}_{F_{\sakne(T)}}$.

   Choose any $T' < T$ and suppose that
$$H_t = 
\bigcup_{t^\frown j\in T'}
\{\beta_{t^\frown j} \}\setminus  
\left({\widehat a}_{F_{t}}\cup \bigcup \{ {\widehat a}_\gamma :
    \gamma\in S_{t^\frown j}\}\right)$$
 is finite for all $t\in T'$.
 We proceed to a contradiction. 
 Since each $F_{t^\frown j}\subset S_t$, there is a finite
 set $I_t\subset S_t$ such that $H_t \subset {\widehat a}_{I_t}$.
Since $I_t\subset S_t$, there is a $\tilde T<_0 T'_t$ forcing
that $I_t\subset \dot S$. Notice then that $I_t\subset S_{t^\frown j}$
for all $t^\frown j\in \tilde T$. It then follows that
$S_{t^\frown j}\cup F_t\cup F_{t_0}$ is a cover for all
 $t^\frown j\in \tilde T$.  But now it follows, by a fusion,
that  there is 
 $T_1 <_0 T'$ satisfying that  $F_{t_0}\cup S_t$ is a cover
for all $t\in T_1$. By Lemma \ref{isacover}, we have
that such a $T_1$ forces that $F_{t_0}\cup\dot S$ is a cover,
contradicting that $T$ forces that $\dot S$ is not an almost cover.
\end{proof}

\begin{theorem} [$\diamondsuit$]
  There is an ultrafilter $\mathcal D$ so that\label{justone} 
  $\mathbb L(\mathcal D)$ preserves that $\mathcal A$ has the ScP.
\end{theorem}

\begin{proof}
Let $\{ Y_\xi : \xi \in \omega_1\}$ be any enumeration of
 $\mathcal P(\omega)$. 
Using standard coding techniques, we may assume we have a sequence
 $\langle \mathcal M_\delta : \omega\leq \delta\in \omega_1\rangle$ with the
 property that, for each $\omega\leq \delta<\omega_1$
 $$\mathcal M_\delta = \langle \langle D(\delta,\alpha)
 : \alpha <
  \delta\rangle , \langle S(\delta,t), F(\delta,t),\beta(\delta,t) :
  t\in \omega^{<\omega}\rangle $$
where each $D(\delta,\alpha)\subset \omega$ and each
   $S(\delta,t)\cup F(\delta,t)\cup\{\beta(\delta,t)\}\subset
   \delta$, 
and satisfying that for each ultrafilter $\mathcal D$ on $\omega$ and
each sequence $\{ S_t , F_t ,  \{\beta_t\} : t\in \omega^{<\omega}\}$ of
subsets of $\omega_1$, there is a stationary set of $\delta$
satisfying that 
\begin{enumerate}
\item  $\mathcal D\supset
           \{ D(\delta,\alpha) : \alpha <\delta\}$,
\item $(S_t\cap \delta ,F_t\cap \delta , \beta_t)= 
       (S(\delta,t), F(\delta,t), \beta(\delta,t))$ for all 
  $t\in \omega^{<\omega}$.
\end{enumerate}

We define a mod finite decreasing
sequence $\langle D_\xi : \xi\in\omega_1\rangle$ by
induction on $\xi\in\omega_1$. For $n<\omega$, 
$D_n = \omega\setminus n$ and choose any infinite
$D_\omega$ so that,  for each $n<\omega$,
  $D_\omega$ is mod finite
 contained in one of $\{Y_n,\omega\setminus
  Y_n\}$. We similarly
  require that for each limit $\delta$ and $n\in\omega$,
 $Y_{\delta+\omega}$ is
  mod finite contained in one of $\{ Y_{\delta+n}, \omega\setminus
   Y_{\delta+n}\}$. These inductive assumptions, if successfully
   completed, ensure that the filter generated by
   $\langle D_\xi : \xi\in\omega_1\rangle$ is an ultrafilter.

   Let $C$ be any cub satisfying that for each $\delta\in C$
   there is a countable elementary submodel $M$ of $H(\omega_2)$
   satisfying that $M\cap \omega_1 = \delta$. Note that each
   $\delta\in C$ is a limit of limit ordinals.

Our induction will proceed along limit ordinals in that for each limit
ordinal $\delta$ we will define the sequence
 $\langle D_\beta : \delta\leq \beta < \delta+\omega\rangle$. If
 $\delta\notin C$ and we have defined the mod finite descending
 sequence
 $\langle D_\xi : \xi < \delta\rangle$, then choose
an infinite  pseudointersection
$D_\delta$ so that $D_\delta$ is also mod finite contained in one
of $\{ Y_\xi , \omega\setminus Y_\xi\}$ for all $\xi<\delta$. For each
$\delta<\beta < \delta+\omega$, we simply set $D_\beta = D_\delta$.

Now assume that $\delta\in C$. Choose, if possible, a countable
elementary submodel $M\prec H(\omega_2)$ so that
\begin{enumerate}
\item $M\cap \omega_1 = \delta$ and $\{ Y_\xi : \xi\in\omega_1\}\in
  M$, 
\item there is an ultrafilter $\mathcal D\in M$ such that
  $\mathcal D \cap M = \langle D_\beta : \beta\in \delta\rangle$
\item there is an $\mathbb L(\mathcal D)$-name $\dot S$ in $M$
  and a $T\in \mathbb L(\mathcal D)\cap M$ forcing that
  $\mathcal A$ does not have the ScP,
\item there is a family $\{ S_t, F_t, \beta_t : t\in T\} \in M$ as in 
  Lemma \ref{thirteen} satisfying that
    \begin{enumerate}
  \item $(S(\delta,t), F(\delta,t), \beta(\delta,t)) =
    (\emptyset,\emptyset,0)$
    for all $t$ such that $t\notin T$
    or $t\subsetneq 
     \sakne(T)$,
   \item for all $t\in T$, $(S_t\cap \delta, F_t,\beta_t)
      = (S(\delta,t), F(\delta,t), \beta(\delta,t))$.
  \end{enumerate}
\end{enumerate}
For definiteness, we may assume we have a well-ordering of
$H(\aleph_2)$ and that $M_\delta$ is chosen to be the minimal (in
this well-ordering) such countable elementary submodel.
Let $D_\delta$ be any infinite pseudointersection
of $\langle D_\xi : \xi<\delta\rangle$ and let
$\{ t_n : n\in \omega\}$ be an enumeration of
$\{ t : \sakne(T)\subseteq t \in T\}$.  By induction on
$n\in\omega$ choose an infinite $D_{\delta+n+1}\subset D_{\delta+n}$
and an $\alpha_{t_n}\in\omega_1$
so that if
$\bigcup_{j\in D_{\delta+n+1}} \{\beta(\delta,t_n^\frown j)\}
\setminus \left( {\widehat a}_{F(\delta,t_n)}\cup
              \bigcup\{{\widehat a}_\gamma : \gamma\in
              S(\delta,t_n^\frown j)\}\right)$ is infinite
            then it converges to 
            some $\alpha_{t_n}$.
            This completes the construction of the ultrafilter
            $\mathcal D$.

            Now assume that $T\in \mathbb L(\mathcal D)$.
            We prove that $T$ does not force
            that $\mathcal A$ fails  the ScP. 
            We consider an  $\mathbb L(\mathcal D)$-name,
            $\dot S$ such that $T$ forces that $\dot S$ is a stationary subset
            of $\omega_1$ and such that
            $\{  \{\beta_t\} \cup F_t : t\in T
            \}\subset[\omega_1]^{<\aleph_0}$  satisfy
            properties (1)-(4) in the
            statement of Lemma \ref{thirteen}.
            We prove that property (5) fails.
            For any $t\in \omega^{<\omega}\setminus T$
             and $t\subsetneq \sakne(T)$,
            set $\beta_t=0$ and $S_t = F_t =\emptyset$. 
We may  choose $\delta\in \omega_1$ so
            that $\mathcal D\supset\{D(\delta,\alpha) : \alpha\in \delta\}$
            and 
$(S_t\cap \delta ,F_t\cap \delta , \beta_t)= 
       (S(\delta,t), F(\delta,t), \beta(\delta,t))$ for all 
  $t\in \omega^{<\omega}$. Let $T_1\in \mathbb L(\mathcal D)$ be any
            condition satisfying that $T_1\subset T$ and
            that for all $t\in T_1$,
            the set $\{ j\in \omega : t^\frown j\}$ is a subset
            of $D_{\delta+\omega}$.
            For each $t\in T_1$, let $\alpha_t$ denote the ordinal
            chosen for $t$ in the stages $\delta$ to $\delta+\omega$
            of
            the construction.
            Let $L$ be the set of $t\in T_1$ satisfying that there is a
             $\gamma_t\in S_t$ with $\alpha_t\in
            \widehat{a}_{\gamma_t}$. We recall that for $t\in L$,
            $t \Vdash_w \gamma_t\in \dot S$. By induction on levels,
            we can now perform a simple fusion to choose
            $T'\leq T_1$ so that for all $t\in T'\cap L$,
            $(T')_t \Vdash \gamma_t\in \dot S$. It then follows
            that, for all $t\in L\cap T'$ and $t^\frown j\in T'$,
 $\gamma_t\in
            S_{t^\frown j}$. 
Now, choose any $\sakne(T')\subseteq
            t\in T'$.         
Let $J= \{ j : t^\frown j\in T'\ \mbox{and}\ 
\beta(\delta,t^\frown j) \notin \left( 
{\widehat a}_{F(\delta,t)}\cup
              \bigcup\{{\widehat a}_\gamma : \gamma\in
              S(\delta,t^\frown j)\}\right)\}$.
              Property              
 (5) will fail if we prove that
$\{ \beta_{t^\frown j} : j\in J\}$ is finite. So we assume that it
is infinite and obtain a contradiction.
Since $S_t\cup F_t\cup F_{\sakne(T)}$ is a cover,
            we may choose $\gamma\in
            S_t\cup F_t\cup F_{\sakne(T)}$ so that
            $\alpha_t\in \widehat{a}_{\gamma}$. Since
            all but finitely many members of
 the set
             $\{ \beta_{t^\frown j} : j\in J\}$
are in $\widehat{a}_{\alpha_t}\cap \widehat{a}_{\gamma}$
and, by (3) of Lemma \ref{thirteen} and the definition of $J$,
 $\widehat{a}_{F_t}\cup \widehat{a}_{F_{\sakne(T)}}$ is
            disjoint from $\{ \beta_{t^\frown j} : j\in
            D_{\delta+n+1}\}$, so it follows
            that $\gamma \in S_t$ and therefore
            that $t\in L$.  However
            this is the contradiction we seek since
            we can assume that $\gamma$ was chosen to be
            $\gamma_t$, and with
            $\gamma_t\in S_{t^\frown j}$ for all $t^\frown j\in T'$,
            $\{ \beta_{t^\frown j} :  j\in J\}$ is disjoint
            from $\widehat{a}_{\gamma_t}$.              
\end{proof}

\section{Basic properties of $\mathbb T$-algebras}

In the previous section we explored how to produce and preserve the
ScP property for a single proper coherent sequence. In this section we
use the concept of $\mathbb T$-algebras to do so for a large system of
proper coherent sequences. These were introduced in \cite{Kosz1}
and utilized in \cite{Roberto, ShelahEfimov} for the purposes of
constructing Efimov spaces.

 Say that $\mathcal A_{\Gamma}$ has the ScP 
    if $\lambda\leq\omega_1$ and $\mathcal A_{\Gamma,x}$ has
    the ScP for all $x\in 2^{\omega_1}$. 

Let us note that the set $X(\Gamma)$ may grow,
 e.g. $\dot X(\Gamma)$ would be a name of all branches,
 if we enlarge
 the model. It may also be useful to note that for $x\in X(\Gamma)$
 (or $x\in \check X(\Gamma)$)
in
the ground model, $(\lambda_x{+}1, \tau(\mathcal A_{\Gamma,x}))$ will
be
unchanged and so 
$\tau(\mathcal A_\Gamma)$ will induce the subspace
topology on the elements of $\check X(\Gamma)$.
Before proving this next result we recall
S. Koppelberg's result \cite{Koppelberg}
that adding a
Cohen real ensures that 
every Stone space of a ground model infinite Boolean
algebra will have a non-trivial converging sequence.
  It is also well-known (as shown in \cite{KunenTall}*{Theorem 9}) 
that a $\sigma$-centered forcing will not add any new
uncountable branches to $\Gamma$ and so $\lambda_x<\omega_1$
for all new $x\in X(\Gamma)$.

\begin{lemma} If $\mathcal A_\Gamma$ is a $\mathbb T$-algebra
  with the ScP and if $G$ is $\mathop{Fn}(\omega,2)$ generic,
  then in $V[G]$,\label{seqcpt}
  $(X(\Gamma), \tau(\mathcal A_\Gamma))$ is sequentially compact.
\end{lemma}

\begin{proof}
  Let $\dot K$ be the $\mathop{Fn}(\omega,2)$ name of
  the closed set of limit points  of an
  infinite discrete subset of $\dot X(\Gamma)$.   
  Clearly the infinite set has a converging subsequence if $\dot K$
  has any points of countable character.
 So   we  may assume
  that  $\dot K$  is an uncountable subset of
   $ \dot{X}(\Gamma)\setminus \dot X(\Gamma,\aleph_0)$,
which is contained in $\check X(\Gamma)$.  By recursion on 
 $\alpha\in
\omega_1$, choose a nice $\mathop{Fn}(\omega,2)$
name $\dot x_\alpha$ for an element of $\dot K\setminus\{\dot x_\beta
: \beta < \alpha\}$. Fix any $p\in G$ satisfying that
$K_p = \{ \alpha \in \omega_1 : (\exists x_\alpha)~~ p\Vdash \dot
x_\alpha = x_\alpha\}$ is uncountable. By Koppelberg's result,
the closure of $K_p\subset \dot K$ 
in $\dot X(\Gamma)$ will contain an infinite
converging sequence. Since $\mathcal A_\Gamma$ retains the ScP,
it follows from Proposition \ref{18} that
this sequence must converge to a new point $x$. However
this is a contradiction since we assumed $\dot K $ is
contained in $\check X(\Gamma)$.
\end{proof}

\section{Preserving the ScP for $\mathbb T$-algebras} 

\begin{theorem}[$\diamondsuit$]
  If $\mathcal A_\Gamma$ is a $\mathbb T$-algebra with the ScP,
  then\label{Dpreserve}
   in the forcing extension by $\mathop{Fn}(\omega_1,2)$,
  there is an ultrafilter $\mathcal D$ satisfying
  that $\mathbb L(\mathcal D)$ preserves that $\mathcal A_{\Gamma}$
  has the ScP.
\end{theorem}

\begin{proof}
  Using a standard coding technique, we can arrange that there is a
  $\diamondsuit$-sequence with the form
  $\vec {\mathcal M} = \{ \mathcal 
M_\delta(i) : i < 5, \delta \in \omega_1\}$ where,
  \begin{enumerate}
  \item $\mathcal M_\delta(0)\in 2^\delta$,
  \item $\mathcal M_\delta(1) = \{ \dot D(\delta,\alpha) : \alpha<\delta\}$
      is a sequence of $\mathop{Fn}(\delta,2)$-names of subsets of
      $\omega$,
      \item $\mathcal M_\delta(2) = \{ \dot S(\delta,t): t\in
        \omega^{<\omega} \}$ and $\mathcal M_\delta(3) =
        \{ \dot F(\delta,t): t\in
        \omega^{<\omega} \}$
are  sequences of
        $\mathop{Fn}(\delta,2)$-names of subsets of $\delta$,
       \item $\mathcal M_\delta(4) = \{ \dot \beta (\delta,t) : t\in
           \omega^{<\omega}\}$
            is a sequence of
            $\mathop{Fn}(\delta,2)$-names of elements  of $\delta$.
      \end{enumerate}
      satisfying that for all $x\in 2^{\omega_1}$,
      sequences
        $\{ \dot D_\alpha :  \alpha\in \omega_1\}$ of nice
        $\mathop{Fn}(\omega_1,2)$ names of subsets of $\omega$,
      all sequences $\{ \dot S_t, \dot F_t : t\in
      \omega^{<\omega}\} $ 
of pairs of nice $\mathop{Fn}(\omega_1,2)$-names  of subsets
$\omega_1$,
and
$\{ \dot \beta_t : t\in \omega^{<\omega}\}$ of nice
$\mathop{Fn}(\omega_1, 2)$-names of ordinals in $\omega_1$,

there is a stationary set of $\delta$ such that

\begin{enumerate}
\item $\mathcal M_\delta(0) = x\restriction\delta$,
\item $\dot D(\delta,\alpha) = \dot D_\alpha$ for $\alpha<\delta$,
\item $\dot S(\delta,t) = \dot S_t$ for $t\in \omega^{<\omega}$,
\item $\dot F(\delta,t) = \dot F_t$ for $t\in \omega^{<\omega}$,
  \item $\dot \beta(\delta,t) = \dot \beta_t$ for $t\in \omega^{<\omega}$.
  \end{enumerate}

  We construct a sequence $\langle \dot D_\alpha : \alpha \in
  \omega_1\rangle $ of nice $\mathop{Fn}(\omega_1,2)$-names of subsets
  of $\omega$ with the intention that $\dot {\mathcal D}$ 
(the filter generated by this sequence) is forced to be an
ultrafilter. Also, that it is forced that $\mathbb L(\mathcal D)$ 
forces that $\mathcal A_{\Gamma}$ has the ScP.

One simple inductive assumption on $\delta<\omega_1$
is that the sequence
$\langle \dot D_\alpha : \alpha < \delta\rangle$ is that 1 forces
that finite intersections are infinite. We may also assume that
each $\dot D_\alpha$ is a nice $\mathop{Fn}(\alpha+\omega,2)$-name.
For $n\in\omega$,
set $\dot D_n$ to be the canonical name for $\omega\setminus n$.
We will emulate the proof of Theorem \ref{justone}. Again, let
$C\subset\omega_1$ be a cub satisfying that for each
$\delta\in C$, there is a countable elementary submodel $M$
of $H(\omega_2)$ satisfying that $M\cap \omega_1=\delta$. Let
$\{ \dot Y_\xi : \xi\in\omega_1\}$ be an enumeration of all
nice $\mathop{Fn}(\omega_1,2)$ names of subsets of $\omega$
enumerated in such a way that $\dot Y_\xi$ is an
 $\mathop{Fn}(\xi,2)$-name. We may
 again inductively 
 assume that for limit $\delta$, $1$ forces, for each $n\in \omega$,
 that
 $    \dot D_{\delta+\omega}$ is mod finite contained in
 one of $\{ \dot Y_{\delta+n} , \omega \setminus \dot
 Y_{\delta+n}\}$. Again, if $\delta$ is a limit not in $C$,
 then $\dot D_{\delta}$ is any nice $\mathop{Fn}(\delta,2)$-name
 of an infinite pseudointersection of the sequence
 $\langle \dot D_\beta : \beta < \delta\rangle$, and for
 $\delta<\beta< \delta+\omega$, $\dot D_\beta = \dot D_\delta$.

 Our main task is accomplished for \textit{critical\/}
values $\delta\in C$. Choose,
 if possible, 
a countable
elementary submodel $M\prec H(\omega_2)$ and
an $x\in X(\Gamma)\cap 2^{\omega_1}$,
again choose the least such pair by some well-ordering,
so that
\begin{enumerate}
\item $M\cap \omega_1 = \delta$, $\{\dot Y_\xi : \xi\in\omega_1\}\in
  M$, and $M_\delta(0)=x\restriction \delta$,
\item there is an
nice $\mathop{Fn}(\omega_1,2)$ name of a
  ultrafilter $\dot {\mathcal D}\in M$ such that 1 forces that
  $\dot{\mathcal D} \cap M \supset \langle \dot 
  D(\delta,\beta) : \beta\in \delta\rangle$\\
(by elementarity,
$\dot {\mathcal D}$ will be unique up to forcing equivalence)
\item there are a nice $\mathop{Fn}(\omega_1,2)$-names
$\dot T$ in $M$ forced to be in $   \mathbb L(\dot{\mathcal D})$ 
and $\dot {\mathcal
    S}$ in $M$ such that 1 forces that $\dot{\mathcal S}$ is a nice  
  $\mathbb L(\dot{\mathcal D})$-name $\dot S$
 that is forced by $\dot T$ to witness that
  $\mathcal A_{\Gamma,x}$ does not have the ScP,
\item there is a family $\{\dot S_t,\dot F_t, \dot \beta_t : t\in
  \dot T\}
  \in M$ that is forced to be
  as in 
  Lemma \ref{thirteen} satisfying that, it is forced for all
  $t\in\omega^{<\omega} $
  \begin{enumerate}
  \item $(\dot S(\delta,t), \dot F(\delta,t), \dot \beta(\delta,t)) =
    (\emptyset,\emptyset,0)$
    if $t\notin \dot T$
    or $t\subsetneq 
     \sakne(\dot T)$,
   \item  $(\dot S_t\cap \delta, \dot F_t,\dot\beta_t)
     = (\dot S(\delta,t), \dot F(\delta,t), \dot \beta(\delta,t)) $
      if  $t\in \dot T$,
    \end{enumerate}
    \item and    
      $\dot S_t = \{ \gamma : t\Vdash_w \gamma\in \dot S\}$ for
        $\sakne(\dot T)\subset t\in \dot T$.
\end{enumerate}
Let $G_\delta$ be a generic filter for $\mathop{Fn}(\delta,2)$.
We complete the construction by working in the forcing extension
  $V[G_\delta]$. For each $\beta <\delta$, fix any $x_\beta\in 
\check X(\Gamma)$ such
that $\varphi_x(x_\beta) = \beta$. 
Let $D_\delta$ be any infinite pseudointersection of
  the sequence $\langle \val_{G}(\dot D_\beta) :\beta
  <\delta\rangle$. Let $\{ t_n : n\in \omega\}$ be an enumeration of
  $\{ t\in \omega^{<\omega} : \sakne(\val_{G}(\dot T))\leq t \in
  \val_G(\dot T)\}$ with $t_0=\sakne(\val_{G}(\dot T))$.
  By removing a finite set, we may assume that
   $\{ t_0^\frown j : j\in D_\delta\}$ is a subset of $\val_{G}(\dot T)$.

  By induction on $n$,
we utilize Lemma \ref{seqcpt} to
  choose an infinite
  $D_{\delta+n+1}$ contained in 
$  \{ j\in D_{\delta+n} : (t_{n+1})^\frown j \in T\}$ so that,
if $$ \bigcup_{j\in D_{\delta+n}} \{\beta(\delta,t_n^\frown j)\}
\setminus \left( {\widehat a}_{F(\delta,t_n)}\cup
              \bigcup\{{\widehat a}_\gamma : \gamma\in
              S(\delta,t_n^\frown j)\}\right)$$
            is infinite,
            then
            $\{ x_\beta : \beta\in B_n\} $ converges
            in $X(\Gamma)$ (to some point of countable character)
            where
            $$B_n =
            \bigcup_{j\in D_{\delta+n+1}} \{\beta(\delta,t_n^\frown j)\}
\setminus \left( {\widehat a}_{F(\delta,t_n)}\cup
              \bigcup\{{\widehat a}_\gamma : \gamma\in
              S(\delta,t_n^\frown j)\}\right).$$
            To finish the construction, simply choose  a sequence
            $\langle \dot D_{\delta+n} : n\in \omega\rangle$
            of nice $\mathop{Fn}(\delta,2)$-names
            that is forced to satisfy the properties of the above
            sequence in $V[G_\delta]$.

            It is time to verify that this works. We have constructed
            the nice name $\dot{\mathcal D}
             = \langle \dot D_\alpha : \alpha \in \omega_1\rangle$ and
             we must verify that 1 forces that
             $\mathbb L(\dot{\mathcal D})$ preserves that
             $\mathcal A_{\Gamma}$ has the ScP. Let
             $G$ be a generic filter for $\mathop{Fn}(\omega_1,2)$
             and,
in $V[G]$,             
             let $\dot S$ be a
             nice $\mathbb L(\val_G(\dot{\mathcal D}))$ name
             of a stationary subset of $\omega_1$ and let
             $x\in \check X(\Gamma)$. Assume that there
             is some $T\in \mathbb L(\val_G(\dot{\mathcal D}))$
             that forces that $\dot S$ is not an almost cover
             for $\mathcal A_{\Gamma,x}$. We may thus assume
             that $T$ has the property that there is a sequence
             $\{\{\beta_t \}\cup F_t , S_t: t\in
             T\}\subset[\omega_1]^{<\aleph_0}$ as in Lemma
             \ref{thirteen}. 
Let $\dot T$ be a nice $\mathop{Fn}(\omega_1,2)$
name for $T$ and similarly let
$\{\dot \beta_t, \dot F_t,\dot S_t : t\in \dot T\}$ be
nice $\mathop{Fn}(\omega_1,2)$ names for the objects as chosen
by Lemma \ref{thirteen}. Choose any $p\in G$ that forces the above
mentioned properties.  Modify if necessary, the above mentioned names
so that if $q\perp p$ is in $\mathop{Fn}(\omega_1,2)$,
$q$ forces that  $\dot \beta_t=0$ and $ \dot F_t, \dot S_t$ are empty.
Similarly, if $q<p$ forces that $t\notin \dot T$
or if $t\subsetneq \sakne(\dot T)$, then $q$ forces that
 $\dot \beta_t=0$ and $ \dot F_t, \dot S_t$ are empty.

Fix any continuous $\in$-chain
 $\{ M^x_\alpha :\alpha\in \omega_1\}$ of countable elementary
 submodels of $H(\aleph_2)$ such that $\{\vec {\mathcal M}, x,p, \{\dot
 Y_\xi: 
 \xi\in\omega_1\}, \dot T,
\{\{\dot \beta_t\}, \dot F_t,\dot S_t : t\in \dot T\}\}
$ is an element of $M^x_0$. Let $C_x$ be a cub such that
$M^x_\gamma\cap \omega_1 = \gamma$ for all $\gamma\in C_x$. Choose
 $\delta\in C_x\cap C$ so that $M^x_\delta$ and these chosen 
names satisfy the requirements of a critical value of $C$.

Let $G_\delta = G\cap \mathop{Fn}(\delta,2)$ and let,
for $n\in \omega$,
            $$B_n =
            \bigcup_{j\in D_{\delta+n+1}} \{\beta(\delta,t_n^\frown j)\}
\setminus \left( {\widehat a}_{F(\delta,t_n)}\cup
              \bigcup\{{\widehat a}_\gamma : \gamma\in
              S(\delta,t_n^\frown j)\}\right)$$
            be as defined in the construction.
            Let $I$ be the set of $n$ such that $B_n$ is infinite,
            and,
for
 each $n\in I$ let
            $y_n\in X(\Gamma)$ be the point such that
             $\{ x_\beta : \beta \in B_n\}$ converges to $y_n$.              
Of course $y_n\neq x$ for all $n$ since no sequence converges to $x$.

We finish our work in $V[G]$. Let $T_\delta <_0 T=\val_G(\dot T)$ be any
extension 
satisfying that for all $t\in T_\delta$,
 the set $\{ j : t^\frown j\in T_\delta\}\subset
 \val_{G}(D_{\delta+\omega})$. Fix any $n$ so that $t=t_n \in T_\delta$
 and let $\alpha = \varphi_x(y_n)$. 
 Since $S_t\cup F_t\cup F_{\sakne(T)}$ is a cover, there is a
 $\gamma_t \in
 S_t\cup F_t\cup F_{\sakne(T)}$ such that $\alpha\in {\widehat
   a}^x_\gamma$ (and $a^x_{\gamma_t}= a_{x\restriction\gamma_t+1}$).
 Since $\{ x_\beta : \beta\in B_n\}$ converges to $y_n$,
 $\{ \beta\in B_n : \varphi_x(x_\beta) = \beta \notin
  {\widehat
    a}^x_\gamma \}$ is finite. By property (3) and (5) of
   Lemma \ref{thirteen}, $B_n$ is disjoint from
   $ {\widehat a}^x_{F_{\sakne(T)}} \cup {\widehat a}^x_{F_t}
   $, hence $\gamma_t\in     S_t$. By definition of
   $S_t$, $t\Vdash_w \gamma_t \in \dot S$. It now follows,
   by a simple fusion, that there is a $T'<_0 T_\delta$
   in $\mathbb L(\mathcal D)$ satisfying that, for all $t = t_n\in T'$
   and $n\in I$, $T'\Vdash \gamma_t \in \dot S$. Furthermore,
   for all $n\in I$ and all $t_n^\frown j\in T'$,
   $\gamma_{t_n} \in S_{t_n^\frown j}$. It now follows
   that the set
$$   
 \bigcup_{t_n^\frown j\in T'} \{\beta(\delta,t_n^\frown j)\}
\setminus \left( {\widehat a}_{F(\delta,t_n)}\cup
              \bigcup\{{\widehat a}_\gamma : \gamma\in
              S(\delta,t_n^\frown j)\}\right)$$
            from (5) of Lemma \ref{thirteen} is a subset of the finite
            set $B_n\setminus {\widehat a}^x_{\gamma_{t_n}}$.
This is our desired contradiction.
\end{proof}

\section{the final model}

We construct some $\mathbb T$-algebras. For $\sigma\in 2^{<\omega}$,
we let $[\sigma] = \{ \tau\in 2^{<\omega} : \sigma\subseteq \tau\}$.
Our first example is a flexible method
of constructing  a $\mathbb T$-algebra on
$\Gamma = 2^{<\omega}$.
Let
 $\{ \sigma_k : k\in \omega\}$ be the standard lexicographic
 enumeration of $ 2^{<\omega}$.

 \begin{proposition} Let  $\pi$
   be any permuation on  $\omega$ and let $L\subset\omega$
   have the property
  that $[\sigma]\cap \{\sigma_k : k\in \pi(L)\}$
 is not empty for all $\sigma\in
  2^{<\omega}$. 
For $\sigma\in 2^{<\omega}$,
  set $a_{\sigma^\frown 0} = \{ k \in\pi(L) : \sigma_k\in
[\sigma^\frown 1]\}$ and
  $a_{\sigma^\frown 1} = \omega\setminus a_{\sigma^\frown 0}$. As
  required,\label{basic} 
   let $a_\emptyset = \emptyset$.
   Then $\{ a_{\sigma} : \sigma \in 2^{<\omega}\}$ is a
     $\mathbb T$-algebra.
\end{proposition}

\begin{proof}
  For each $\sigma\in 2^{<\omega}$, let
  $[\sigma]_L $
  denote the set
  $\{ k \in \pi(L) : \sigma\subseteq \sigma_k\}$. 
  Therefore  $a_{\sigma^\frown 0} = [\sigma^\frown 1]_L$
  and it follows that
 $a_\sigma \cap [\sigma]_L$ is empty for all
  $\sigma\in 2^{<\omega}$. This ensures that $a_{\sigma}\setminus
  \bigcup \{ a_{\sigma\restriction j} : j<|\sigma|
  \}$ is infinite for all $\emptyset \neq \sigma\in 2^{<\omega}$.
  It is also now evident that $a_{\sigma^\frown 0}$ is disjoint from
  $a_{\sigma\restriction j}$ for all $j\leq |\sigma|$.
  This implies that $a_{\sigma^\frown 1}$ contains
    $a_{\sigma\restriction j}$ for all $j\leq |\sigma|$. This
  completes
  the proof.
\end{proof}

We can lift this simple construction to extend any
$\mathbb T$-algebra $\mathcal A_{\Gamma}$ so long as
there are $x\in X(\Gamma)$ with $\lambda_x < \omega_1$.
Fix a sequence $\{ e_\lambda : \omega\leq \lambda\in \omega_1\}$ where,
for each $\lambda\in\omega_1$, $e_\lambda $ is a bijection from
$\omega$ to $\lambda$. For convenience, $\Gamma$ will always
denote some non-empty subtree of 
$ 2^{<\omega_1}$ that is twinned and has no maximal elements.
For definiteness, we assume for the remainder of this section,
that for every $\mathbb T$-algebra $\mathcal A_{\Gamma}$,
the family $\{ a_\sigma : \sigma \in 2^{<\omega}\}$ is obtained as
in Proposition \ref{basic} when $\pi$ is the identity map
and $L =\omega$.

\begin{definition}
Let $\mathcal A_{\Gamma}$ be a $\mathbb T$-algebra\label{nextstep} 
  and let
$x\in X(\Gamma, \aleph_0 )$ and let $\pi$ be a permutation on
$\omega$. We define:
\begin{enumerate}
\item 
  $\Gamma^x = \Gamma \cup \{ x^\frown \sigma : \sigma\in
  2^{<\omega}\}$,
  \item  recursively
    for  $n\in\omega$, let $c^{x}_n = 
  a^x_{e_{\lambda_x}(n)} \setminus \bigcup_{k<n} c^{x}_k$,
\item    $L_x = \{ n : c^{x}_n \neq \emptyset\}$,
\item  $a^{\pi,x}_x = \emptyset$
  and 
$a^{\pi,x}_{\sigma ^\frown 0} = \bigcup\{ c^{x}_k :
     \sigma_{\pi(k)}\in [\sigma^\frown 1]\}$
for $\sigma\in 2^{<\omega}$, 
\item $\mathcal A_\Gamma[\pi,x]$ is the collection
    $\mathcal A_\gamma \cup \{ a^{\pi,x}_\sigma : \sigma\in
    2^{<\omega}\}$,
  \item for all $Y\subset X(\Gamma,\aleph_0)$, $\Gamma^Y
    =\bigcup\{\Gamma^y : y\in Y\}$, and\\
    $\mathcal A_\Gamma[\pi, Y] =
    \bigcup \{ \mathcal A_\Gamma[\pi,y] : y\in Y\}$.
\end{enumerate}
\end{definition}

One could also define $\mathcal A_\Gamma[\langle \pi_y, y: y\in
Y\rangle]$ to equal $\bigcup\{ \mathcal A_\Gamma[\pi_y,y] : y\in Y\}$
but we will not need this.   

\begin{lemma} The collection $\mathcal A_\Gamma[\pi,x]$
is a $\mathbb T$-algebra on $\Gamma^x$ so long as
$[\sigma]\cap \{ \sigma_k : k\in \pi(L_x)\}$ is not empty
for all $\sigma\in2^{<\omega}$. 
\end{lemma}

We skip the  proof of this next result since the
only non-immediate property in the definition
of being a $\mathbb T$-algebra depends only on the behavior
of each $\mathcal A_{\Gamma,x}$.

\begin{proposition} Let $Y\subset X(\Gamma,\aleph_0)$ and let
  $\pi$ be a permutation on $\omega$.
  $\mathcal A_\Gamma[\pi, Y]$ is a $\mathbb T$-algebra on $\Gamma^Y$
  so long
   as $\mathcal A_\Gamma[\pi,y]$ is a $\mathbb T$-algebra for all
   $y\in Y$.   
 \end{proposition}

Let $Q_{perm}$ denote the poset consisting of functions $\psi\in
\omega^{<\omega}$ that are  1-to-1. 
$Q_{perm}$ is ordered by extension and is 
forcing isomorphic to $\mathop{Fn}(\omega,2)$. If $G$ is the generic
for $Q_{perm}$, then $\pi_G = \bigcup G$ is a permutation on $\omega$.

\begin{lemma} Let $\mathcal A_{\Gamma}$ be a $\mathbb T$-algebra
  and let $x\in X(\Gamma,\aleph_0)$ and let
  $\{x_n : n\in\omega\}\subset X(\Gamma)$ converge to $x$
  in $\tau(\mathcal A_{\Gamma})$.
If  $G$ is  a $Q_{perm}$ generic filter
and $\pi = \bigcup G$,
then $\mathcal A_\Gamma[\pi,x]$
is a $\mathbb T$-algebra on $\Gamma^x$ and
 $\{x_n:n\in \omega\}$ does not
converge in $\tau(\mathcal A_\Gamma[\pi,x])$.
\end{lemma}

Here is where Cohen genericity will help create
a canonical $\mathbb T$-algebra
that has the ScP.

\begin{lemma} 
  Let $\langle P_\alpha, \dot Q_\alpha : \alpha<\mu\rangle$ be a
  finite support iteration of $\sigma$-centered posets.\label{canonical} 
  Assume that $E$ is a set of limit ordinals that is cofinal
in $\mu$ and that
  $\dot Q_\alpha$ is the canonical $P_\alpha$ name for
 $Q_{perm}$ for all $\alpha\in E$.
  Let $\mathcal A_\Gamma$ be a $\mathbb T$-algebra
  and let $G\subset P_\mu$ be a generic filter. For each
  $\alpha<\mu$, let $G_\alpha = G\cap P_\alpha$ and
  let $\pi_\alpha$ be the permutation obtained 
  from $G_{\alpha+1}\cap \val_{G_\alpha}(\dot Q_\alpha)$. 

\noindent
In $V[G]$, we recursively define $\langle
\Gamma_\alpha,\Gamma_{\alpha+1} : \alpha \in 
E\rangle $ and $\langle
\mathcal A_{\Gamma_\alpha} ,\mathcal A_{\Gamma_{\alpha+1}} 
: \alpha\in E\rangle$:
\begin{enumerate}
  \item $\Gamma_\alpha = \Gamma\cup \{ \Gamma_{\beta+1} : \beta\in
    E\cap \alpha\}$
and
$\mathcal A_{\Gamma_\alpha} = \mathcal A_{\Gamma}\cup
      \bigcup \{ \mathcal A_{\Gamma_{\beta+1}}  : \beta \in E\cap \alpha\}$,
    \item $\Gamma_{\alpha+1} = \Gamma_{\alpha}^{Y_\alpha}$
      where $Y_\alpha = X(\Gamma_\alpha,\aleph_0)\cap V[G_\alpha]$,
      \item $\mathcal A_{\Gamma_{\alpha+1}}$ is defined as
$\mathcal A_{\Gamma_\alpha}[{\pi_\alpha},Y_\alpha]$.
\end{enumerate}
Then, in $V[G]$,
$\mathcal A_{\Gamma_\mu}$ is a $\mathbb T$-algebra,
 and $\mathcal A_{\Gamma_\mu}$ 
has the ScP so long as  each $\mathcal A_{\Gamma_\alpha}$ has the ScP.
\end{lemma}

\begin{proof}
  We leave the routine verification that $\mathcal A_{\Gamma_\mu}$ is
  a $\mathbb T$-algebra as an exercise.
Assume that $\mathcal A_{\Gamma_\alpha}$ has the ScP for all
$\alpha\in E$. 
Fix any $x\in X(\Gamma_\mu)$ with $\lambda_x=\omega_1$.
We must prove that   $\mathcal A_{\Gamma_\mu,x}$ has the ScP.
It is immediate that if $x\in X(\Gamma_\alpha)$ 
for some $\alpha\in E$, then
$\mathcal A_{\Gamma_\mu,x} = \mathcal A_{\Gamma_\alpha,x}$
    does have the ScP. Therefore we may assume that
    $x\notin \bigcup_{\alpha\in E}X(\Gamma_\alpha)$.
    Since $\Gamma_\mu = \bigcup_{\alpha\in E}\Gamma_\alpha$, we
    may choose $\gamma_\alpha\in \omega_1$, for each
 $\alpha\in E$,  so that
    that $x\restriction \gamma_\alpha \in X(\Gamma_\alpha)$.
 Note that  $\langle \gamma_\alpha : \alpha\in E\rangle$ is monotone
    increasing and cofinal in $\omega_1$. Therefore the proof is
    complete if $\mu$ does not have cofinality $\omega_1$. 
    
    Now we return to $V$ and argue by forcing.
    For each $\alpha\in E$, let $\dot \Gamma_\alpha$ be a
    nice name for $\Gamma_\alpha$. 
    Fix a nice
     $P_\mu$-name for $\dot x$ and $p\in P_\mu$ forcing
     that $\dot x \in 2^{\omega_1}$ and, for each $\alpha\in E$,
     let $\dot \gamma_\alpha$ be a nice $P_\mu$ name so that
      $p\Vdash \dot x\restriction \dot\gamma_\alpha \in X(\dot
      \Gamma_\alpha)$. Since $P_\mu$ is ccc and $\dot \gamma_\alpha$
      is a nice name, there is a minimal $\beta_\alpha<\mu$ such
      that $\dot x\restriction \dot \gamma_\alpha$ is a $P_{\beta_\alpha}$ name. 
      This implies that 1 forces that
       $\dot x\restriction \dot \gamma_\alpha \in
       \dot\Gamma_{\beta_\alpha+1}$. 
      Fix a cofinal sequence $\{ \alpha_\xi : \xi\in \omega_1\}\subset E$ 
so that $\alpha_\eta$ and $\beta_{\alpha_\eta}$ are less than
$\alpha_\xi$ for all $\eta<\xi$. For each limit $\delta\in \omega_1$,
 let $\mu_\delta\leq \alpha_\delta $ 
be the supremum of $\{\alpha_\eta : \eta<\delta\}$. 
Note also that for each limit $\delta$, 
$\dot x\restriction\dot \gamma_{\alpha_\delta}$
is a $P_{\mu_\delta}$-name since it is forced to equal
the union  of $\{ \dot x\restriction
\dot \gamma_{\alpha_\eta} : \eta < \delta\}$.

      Now also let
       $\dot S$ be a nice $P_\mu$-name for a stationary subset of
       $\omega_1$. There is a stationary set $S_1$ such that for
       $\delta\in S_1$, there is a $p_\delta < p$ forcing that
       $\delta \in \dot S$ and a $\gamma_\delta$ such
that  
 $p_\delta\restriction\mu_\delta\Vdash \dot
 \gamma_{\alpha_\delta} = \gamma_\delta$. We may also assume
that $p_\delta$ decides the value of $\dot x(\gamma_\delta+1)$.
 By the pressing down lemma, we can assume
        that there is some $\zeta<\mu$ such that
           $\mathop{supp}(p_\delta)\cap \mu_\delta\subset \zeta$ for
           all
           $\delta\in S_1$. Note that $p\in P_\zeta$ and,
            as is well-known, 
we can choose 
            a $P_\zeta$-generic $G_\zeta$ with $p\in G_\zeta$
so             that, in $V[G_\zeta]$,
 $S = \{ \delta\in S_1 : p_\delta\restriction \zeta\in
 G_\zeta\}$ is stationary. We continue our work
 in $V[G_\zeta]$. By passing to a stationary subset, we
 can assume that for all $\delta<\gamma$ in $S$,
  $\dom(p_\delta)\subset \mu_\gamma$.
A final
  reduction on $S$ is to assume that there is a single 
$\psi\in Q_{perm}$
  and function $\rho\in  \omega_1^m$
  so that, for all $\delta\in S$, $p_\delta(\alpha_\delta) =\psi$
 and $\rho = e_{\lambda_{\gamma_\delta}}\restriction \dom(\psi)$.

Let
$F = \{ \rho(k) : k\in \dom(\psi)\}$. We are ready to prove
that there is some $q\in G_\zeta$ forcing that $F\cup \dot S$ is a
cover.  Choose $\bar p\in P_\mu$ such that $\bar p\restriction \zeta\in
G_\zeta$ and any $\xi\in \omega_1$. We may assume that $\bar p \Vdash
\xi \notin {\widehat a\,}^{\dot x}_F$. Fix any $\delta\in S$ so
that $\xi<\gamma_\delta$ and
$\bar p\in P_{\mu_\delta}$. Now jump first to $V[G_{\mu_\delta}]$
with $\bar p\in G_{\alpha_\delta}$. Recall that $p_\delta\restriction
\mu_\delta\in G_\zeta\subset G_{\mu_\delta}$, so we may assume
that $\bar p<p_\delta$.
  Now pass further to
 $V[G_{\alpha_\delta}]$ so that $p_\delta\restriction \alpha_\delta
\in G_{\alpha_\delta}$. 
We now prove there is an extension
$q$ of $p_\delta$ (with
 $q\restriction \mu_\delta<\bar p$) that forces
that $\xi \in {\widehat a\,}^{\dot x}_\delta$.
For simpler notation we prove just
the case when $p_\delta$ forces that $\dot x(\gamma_\delta+1)=0$. 
Let $\pi$ denote $\dot \pi_{\alpha_\delta}$ and note
that $p_\delta$ forces (only)
that $\psi\subset \pi$. 

Let $y$ denote $\val_{G_{\alpha_\delta}}(\dot x\restriction \gamma_\delta)$.
Let $\langle a_\eta : \eta < \gamma_\delta\rangle$ denote
the sequence $\langle 
\val_{G_{\alpha_\delta}}(a^{y}_\eta): \eta < \gamma_\delta\rangle$. 
These are infinite subsets of $\omega$
and are also in $V[G_{\mu_\delta}]$. Also let $\lambda$ denote
$\lambda_{\gamma_\delta}$.  
Set $c^{y}_n =
  a^{y}_{e_\lambda(n)} \setminus
  \bigcup_{\ell<n} c^y_\ell$ for all $n\in\omega$ (as in
  Definition \ref{nextstep}).
Choose $j$ minimal so that
 $\xi \in {\widehat a\,}^{y}_{e_\lambda(j)}$. 
Since $F $ is the initial segment 
$\{e_{\lambda}(n)  : n \in\dom(\psi)\}$, it follows that
${\widehat a}_F = \bigcup_{k\in\dom(\psi)} c^{y}_k$ and so
 $j$ is not in the domain  of $\psi$.
Recall that $a^{\dot x}_{\delta}
= a^{\pi,x_\delta}_{\langle 0\rangle}$ is defined to
 equal $\bigcup \{ c^{y}_k : \sigma_{\pi(k)} \in [\langle
1\rangle]\}$.   Choose any $q<p_\delta
 $ with $q\restriction \alpha_\delta \in G_{\alpha_\delta}$
so that $
q(\alpha_\delta)= \tilde \psi \supset \psi
=p_\delta(\alpha_\delta)$ is
an element of $Q_{perm}$ satisfying
 that $\sigma_{\tilde \psi(j)} \in [\langle1\rangle]$. It follows from
 the minimality of $j$ that $\xi \in
  {\widehat a\,}^y_{e_\lambda(j)} 
\setminus \bigcup \{ {\widehat a\,}^y_{e_\lambda(\ell)} : \ell <
  j\}$.  Since $q$ forces that 
$c^y_j = 
  {a\,}^y_{e_\lambda(j)} \setminus \bigcup \{ { a\,}^y_{e_\lambda(\ell)} : \ell <
  j\}$ is a subset of $a^{\dot x}_\delta$, it should be clear
  that $q$ forces that
   ${\widehat a\,}^{\dot x}_\delta$ contains
$  {\widehat a\,}^y_{e_\lambda(j)} 
\setminus \bigcup \{ {\widehat a\,}^y_{e_\lambda(\ell)} : \ell <j \}$,
 and therefore $\xi$ as required.
\end{proof}

Now the main theorem.

\begin{theorem}
  It is consistent that there is an Efimov space with\label{maintheorem}
 character
  $\aleph_1$ while the splitting number is $\aleph_2$.
  In particular, it is consistent with $\mathfrak s =\aleph_2$
to have a $\mathbb T$-algebra on $2^{<\omega_1}$ whose Stone space is
an Efimov space.
\end{theorem}

\begin{proof}
  We assume that $\diamondsuit$ and $2^{\aleph_1}=\aleph_2$ holds
  and we produce a ccc poset to prove the consistency of the
  statement. Let $S^2_1$ denote the ordinals in $\omega_2$
  that have uncountable cofinality. Let
  $(S^2_1)'$ denotes the limit points of $S^2_1$ in $\omega_2$
  and    let $E = S^2_1\cap  (S^2_1)'$ (i.e. $E$ is the limit points of
  uncountable cofinality). 
  Fix a well-ordering $\sqsubset$ of $H(\aleph_2)$.
Let $\Gamma_\omega$ be $2^{<\omega}$ and $\mathcal
      A_{\Gamma_\omega}$ be as defined immediately following Proposition
      \ref{basic}.

  We define a finite support iteration $\langle P_\alpha , \dot
  Q_\alpha : \alpha\in \omega_2\}$ of $\sigma$-centered posets by a
  recursion on $\alpha$. We also let 
  $\dot \Gamma_\alpha, \dot \Gamma_{\alpha+1},
  \dot{ \mathcal A}_{\Gamma_\alpha}, \dot {\mathcal
    A}_{\Gamma_{\alpha+1}}$ be recursively 
   defined as in Lemma \ref{canonical},
  \begin{enumerate}
  \item for $\alpha\notin  S^2_1$, $\dot Q_\alpha$ is the
    $P_\alpha$ name for $\mathop{Fn}(\omega,2)$,
  \item for $\alpha\in E$, $\dot \Gamma_\alpha$ is a
     $P_\alpha$ name for $2^{<\omega}\cup\bigcup_{\beta\in E\cap \alpha} \dot
    \Gamma_{\beta+1}$,

  \item for $\alpha\in E$, $\dot Y_\alpha$ is
    the $P_\alpha$ name for $X(\dot\Gamma_\alpha,\aleph_0)$
    and 
    $\dot \Gamma_{\alpha+1} = \dot\Gamma_{\alpha}^{\dot Y_\alpha}$
(as in Definition \ref{nextstep},

  \item for all $\alpha\in E$, $\dot Q_\alpha$ is the $P_\alpha$ name
    for $Q_{Perm}$ and $\dot \pi_\alpha$ is the $P_{\alpha+1}$ name
    for the permutation added by $\dot Q_\alpha$,

  \item for $\alpha\in E$, $\dot {\mathcal A}_{\Gamma_{\alpha}}$
    is the $P_\alpha$ name for $\mathcal A_{\Gamma_\omega}\cup
     \bigcup\{ \dot {\mathcal A}_{\Gamma_{\beta+1}} : \beta\in E\cap
     \alpha\}$
     and $\dot {\mathcal A}_{\Gamma_{\alpha+1}}$ is the $P_{\alpha+1}$
     name
     for $\dot {\mathcal A}_{\Gamma_\alpha}[\pi_\alpha,\dot
       Y_\alpha]$,
and finally,
\item for $\alpha \in S^2_1\setminus (S^2_1)'$
  and $\mu_\alpha = \sup(E\cap \alpha)$, 
  $\dot{\mathcal D}_\alpha$ is the $\sqsubset$-least $P_\alpha$-name of
  an ultrafilter on $\omega$ such that $P_\alpha*\mathbb L(\dot
  {\mathcal D}_\alpha)$ forces that
  $\dot {\mathcal A}_{\Gamma_{\mu_\alpha}}$ has the ScP
  (if one exists),
  and $\dot Q_\alpha$ is the $P_\alpha$ name for $\mathbb L(\dot
  {\mathcal D}_\alpha)$. If no such $\dot {\mathcal D}_\alpha$ exists, 
   then $\dot Q_\alpha$ is the $P_\alpha$ name for $\mathop{Fn}(\omega,2)$.
  \end{enumerate}

  \begin{claim} For each $\mu\in (S^2_1)'\cup\{\omega_2\}$, $P_\mu$
    forces that $\dot{\mathcal A}_{\dot \Gamma_\mu}$ is a $\mathbb
    T$-algebra with the ScP.
  \end{claim}

  \bgroup
  \def\proofname{Proof of Claim:}

  \begin{proof} We prove the Claim by induction on $\mu$. 
If $E\cap \mu$ is cofinal in $\mu$, then this follows from 
the inductive assumption and Lemma \ref{canonical}. 
  Otherwise, let $\alpha=\sup(E\cap \mu)$ and we
     break into two
  cases. 
  In the first case $\alpha\in E$ and $\dot {\mathcal A}_{\dot \Gamma_\mu}$
  is just the $P_{\alpha+1}$ name for $\dot {\mathcal A}_{\dot \Gamma_{\alpha+1}}$. 
  Since $\dot {\mathcal A}_{\dot \Gamma_\alpha}$ is assumed to have the ScP
  in the forcing extension by $P_\alpha$,
  and since $\dot \Gamma_{\alpha+1} \setminus \dot \Gamma_\alpha $ is
  forced to be a subset of
 $ 2^{<\omega_1}$ , it follows that 
     $\dot {\mathcal A}_{\dot \Gamma_\mu}$ is forced to have the ScP. Now assume
     that $\alpha\notin E$. It then follows that $\alpha\in
     S^2_1\setminus (S^2_1)'$.  
     In this case, $\dot {\mathcal A}_{\Gamma_\mu}$ is forced to equal
       $\dot {\mathcal A}_{\Gamma_\alpha}$ (since $E\cap [
       \alpha,\mu)
       $ is empty) and so we must verify that $P_{\mu}$ preserves
       that $\dot{\mathcal A}_{\Gamma_\alpha}$ has the ScP.
       By the inductive assumption, $P_\alpha$ forces
       that $\dot {\mathcal A}_{\Gamma_\alpha}$ has the ScP
         and $ \mathop{Fn}(\omega,2)$ always preserves the ScP
         (Lemma \ref{lemma7}).
 Therefore, by   the choice of $\dot Q_\alpha$ in clause (6),
      $P_{\alpha+1}$ forces that $\dot {\mathcal A}_{\Gamma_\alpha}$ 
      has the ScP. For all $\alpha < \beta < \mu$, 
        $\dot Q_\beta$ is chosen to be $\mathop{Fn}(\omega,2)$,
        and so, again by Lemma \ref{lemma7}, $P_\mu$ forces
        that $\dot {\mathcal A}_{\dot \Gamma_\alpha}
         = \dot {\mathcal A}_{\dot \Gamma_\mu}$ has the ScP.
  \end{proof}
  
\begin{claim}
 $P_{\omega_2}$ forces that $\mathfrak s = \aleph_2$.
\end{claim}

\begin{proof}
 Since $P_{\omega_2}$ is ccc, it suffices to prove that 
 for all $\mu\in S^2_1\setminus (S^2_1)'$, $P_{\mu}$
 forces that $\dot Q_\mu$ adds an unsplit real. 
 Fix any $\mu \in S^2_1\setminus (S^2_1)'$
 and let $\alpha = \sup(E\cap \mu)< \mu$. 
It is
 well-known that  the generic real added by
 $\mathbb L(\mathcal D)$ is unsplit providing $\mathcal D$ is 
 an ultrafilter on $\omega$. By Claim 1, $P_\mu$ forces
 that $\dot{\mathcal A}_{\dot \Gamma_{\alpha+1}}$ has
 the ScP. 
Let $G_{\alpha +1}$ be any $P_{\alpha+1}$-generic filter
and let 
$\mathcal A_{\Gamma_{\alpha+1}}$ (which has the ScP)
 be the valuation by $G_{\alpha+1}$ of 
   $\dot{\mathcal A}_{\dot \Gamma_{\alpha+1}}$.
 Since $\diamondsuit$ is assumed to hold in the ground
 model and $P_{\alpha+1}$ is a ccc poset of cardinality $\aleph_1$,
$\diamondsuit$ also holds in    $V[G_{\alpha+1}]$ 
(\cite{Kunen}{VII.H.8}). By Theorem \ref{Dpreserve}, 
 there is a $\mathop{Fn}(\omega_1,2)$  name, $\dot {\mathcal D}$ 
 satisfying that $\mathop{Fn}(\omega_1,2)*\mathbb L(\dot{\mathcal D})$ 
 preserves that $\mathcal A_{\Gamma_{\alpha+1}}$ has the ScP.
    Since  the forcing extension by $P_{\mu}$ is equal to the
    forcing extension of $V[G_{\alpha+1}]$ by
     $\mathop{Fn}(\omega_1,2)$, it follows that 
     there is a $\dot D_{\mu}$ as required in clause (6) of 
     the construction so that $P_\mu$ forces
     that $\dot Q_\mu$ is $\mathbb L(\mathcal D)$ for an ultrafilter 
      $\mathcal D$.
\end{proof}

  \egroup

This completes the proof of the Theorem.
\end{proof}

\section{Splitting number and Scarborough-Stone problem}

The Scarborough-Stone problem asks if each product of sequentially
compact spaces is countably compact. A space is sequentially compact
if every infinite sequence has a converging subsequence. Just as
with the Efimov space question, this problem is not known to be
independent of ZFC. With this, the original formulation of the
problem, it is not known if an affirmative answer is consistent with
ZFC. 
By the
results in \cite{ShelahEfimov}, we have the following corollary to
Theorem \ref{maintheorem}.  

\begin{corollary} It is consistent that there is a family of
first-countable\label{smallsize} 
  sequentially compact spaces of cardinality less than $\mathfrak s$
  whose product is not countably compact. 
\end{corollary}
 
We give a brief sketch of the proof.

\begin{proof} Let $\Gamma = 2^{<\omega_1}$ and 
let $\mathcal A_\Gamma = \{ a_\sigma : \sigma \in
  2^{<\omega_1}\}$ be the $\mathbb T$-algebra constructed in
  Theorem \ref{maintheorem} in the model in which
  $\mathfrak s =\aleph_2$. For each $x\in 2^{\omega_1}$, let
  $\mathcal A_{\Gamma,x}$  be the proper coherent sequence as defined
  in Definition \ref{fifteen} and let
  $Y_x$ denote the topology on $\omega_1$ as a subspace of
  of $\omega_1+1 = \lambda_x+1$ as defined in Proposition \ref{two}.
  Also let $\varphi_x$ be the continuous
mapping from $X(\Gamma)$ onto
$Y_x\cup\{\omega_1\}$ as defined in Lemma \ref{sixteen}. It is evident
from Proposition \ref{two} that $Y_x$ is first-countable.
We prove
that $Y_x$ is sequentially compact. Fix any infinite sequence
$\{ y_n : n\in\omega\}\subset Y_x$ and choose any sequence
$\{x_n : n\in \omega\}\subset X(\Gamma)$ satisfying that
$\varphi_x(x_n)=y_n$ for all $n\in \omega$. Since $X(\Gamma)$ is
compact and 
has no converging sequences, there is some limit $z\in X(\Gamma)$
of $\{x_n :n\in \omega\}$ with $z\neq x$. Let $y=\varphi(z)\in Y_x$ and note
that $y $ is a limit of $\{ y_n : n\in \omega\}$. Since $y$ has
countable character in $Y_x$, $\{y_n :n\in\omega\}$ has a subsequence
converging to $y$. Now we prove that the product of the family
$\{ Y_x : x\in 2^{\omega_1}\}$ is not  countably compact. Choose any
$r\in 2^\omega$ and consider the product of the subfamily
$\{ Y_x : r\subset x \in 2^{\omega_1}\}$. For each $n\in \omega$,
let $\overline{n}$ denote the function in the product
$\Pi\{ Y_x : r\subset x \in 2^{\omega_1}\}$ that has value
$n$ in every coordinate (recall that $Y_x$ has base set
$\omega_1$). Let
$\vec y$ be any point in
$\Pi\{ Y_x : r\subset x \in 2^{\omega_1}\}$ and we show
that $\vec y$ is not a limit point of
 $\{ \overline{n} : n \in\omega\}$. Since each $n\in\omega$ is
isolated in each $Y_x$, we may assume that $\vec y(x)\geq\omega$ for
all $r\subset x\in 2^{\omega_1}$. 
For each $\sigma\in 2^{<\omega_1}$,
let $X_\sigma = \{x \in 2^{\omega_1} : \sigma\subset x\}$. For each
$x\in X_r$, let $\sigma_x =x\restriction (y(x)+1)$ and
note that $ \widehat a_{\sigma_x}$ is a
  neighborhood of $y(x)$ in $Y_x$.  Let 
$W_x$ be the product neighborhood
  $\pi_x ^{-1}[ \widehat a_{\sigma_x}]$ in
    $\Pi\{ Y_x : x\in X_r\}$. By Definition \ref{fifteen} (4),
    ${a}_{\sigma_x}$ and ${a}_{\sigma_x^\dagger}$ are
    disjoint and by the coherence property, so
    are        $\widehat{a}_{\sigma_x}$ and
    $\widehat{a}_{\sigma_x^\dagger}$. If there are $x,x_1 \in
    X_r$ such that $\sigma_{x_1} = \sigma_x^\dagger$,
    then $W_x \cap W_{x_1} \cap \{ \overline{n} : n\in\omega\}$
    is equal to $\{ \overline{n} : n\in
 \widehat{a}_{\sigma_x}\cap
 \widehat{a}_{\sigma_x^\dagger}\}$, and so is empty.
 Therefore we consider the case that for all $r\subset \sigma\in
 2^{<\omega_1}$, one of $\{\sigma^\frown 0, \sigma^\frown 1\}$ is not
 in the set $\{ \sigma_x : x\in X_r\}$. Beginning with $\sigma_0 = r$,
 we can now recursively construct a strictly increasing sequence
 $\{ \sigma_\xi : \xi\in\omega_1\}\subset 2^{<\omega_1}$ satisfying
 that for each $\xi\in\omega_1$ and each $x\in X_{\sigma_\xi}$,
  $\sigma_{\xi}\subsetneq \sigma_x$. Of course $x =\bigcup\{
 \sigma_\xi : \xi\in\omega_1\}$ contradicts the assumption
 that $y(x)<\omega_1$. 
\end{proof}

In the other direction, it is a natural question to ask if
a negative answer to the Scarborough-Stone problem
follows from $\mathfrak s=\aleph_1$. Many 
partial results are
known, see, for example,
Vaughan's survey article
\cite{VaughanSS}. It is shown in \cite{HrusakIsrael} that the
assumption $\diamondsuit(\mathfrak s)$, a strengthening of
$\mathfrak s=\aleph_1$, suffices. In the spirit of Corollary
\ref{smallsize}, we pose the following problems.

\begin{enumerate}
\item Does $\mathfrak s=\aleph_1$ imply a negative answer to the
  Scarborough-Stone problem?
\item Does $\mathfrak s=\aleph_1$ imply there is a family of
   sequentially compact spaces, each of cardinality at most $
   \aleph_1$, whose product is not countably compact?
   \item Does $\mathfrak s = \aleph_1$ imply there is a $\mathbb
     T$-algebra on $2^{<\omega_1}$ whose Stone space has no converging
     sequences?
\end{enumerate}

Two well-known strengthenings of the assumption $\mathfrak s=\aleph_1$
may be relevant to these questions. A splitting family $\mathcal S $
is $\aleph_0$-splitting if for every countable family
 $\{ b_n : n\in\omega\}$ of infinite subsets of $\mathbb N$, there is
a single member of $\mathcal S$ that splits each of them. It is not
known if an $\aleph_0$-splitting family of cardinality $\mathfrak s$
necessarily exists. A splitting family
$\mathcal S = \{ s_\alpha : \alpha\in \mathfrak s\}$  is
tail-splitting if for each infinite $b\subset \mathbb N$, the
set of members of $\mathcal S$ that split $b$ contains a 
final segment of $\mathcal S$. It is known to be consistent that there
is no tail-splitting sequence of cardinality $\mathfrak s$. 
We formulate a still stronger condition that is sufficient to
obtain the conclusions of problems (1)-(3). This condition will hold
in a forcing extension by uncountably many Random reals, or by a
finite support iteration with cofinality equal to $\omega_1$.
Recall that $H(\omega_1)$ is equal to the set of sets whose
transitive closure is countable.  Say that $M\subset H(\omega_1)$ is a
model if it, i.e. $(M,\in)$,
is a model of all the axioms of ZF with the exception of
the power set axiom. Of course $H(\omega_1)$ itself is a model
 (see \cite[IV]{Kunen}). 

\begin{proposition} Assume that $H(\omega_1)$ can written as an
  increasing chain $\{ M_\xi  : \xi\in\omega_1\}$ of models in such
  a way that for each $\xi $, there is a subset of $\omega$
  that splits every member of $M_\xi\cap [\omega]^\omega$,
  then there is a $\mathbb T$-algebra on $2^{<\omega_1}$ whose Stone
  space has no converging sequences. 
\end{proposition}

The proof is a minor variant of similar proofs in
\cite{Kosz1,ShelahEfimov}. 

\begin{proof} 
  For each $\alpha\in\omega_1$, let $\Gamma_\alpha = 2^{<\alpha}$.
  Also, for each $\omega \leq \alpha\in\omega_1$,
  let $e_\alpha :\omega\rightarrow \alpha$ be a bijection
  onto the successor ordinals in $\alpha$. 
  Let $\{ a_\sigma : \sigma\in \Gamma_\omega\}$ be the
  $\mathbb T$-algebra as defined in Definition \ref{nextstep}.
  For each $x\in 2^\omega$, fix a $\xi_x\geq \omega$ so that
   $e_\omega$ and $\{ a_{x\restriction n} : n\in \omega\}$ are in
  $M_{\xi_x}$. 
  By induction on $\omega \leq \alpha<\omega_1$ we construct
  a $\mathbb T$-algebra 
$\mathcal A_{\Gamma_\alpha} = \{ a_\sigma^\alpha : \sigma\in \Gamma_\alpha\}$ 
and choose ordinals $\{ \xi_\sigma : \sigma\in \Gamma_\alpha\}$ so
that the following induction hypotheses
hold for all $\sigma\in \Gamma_\alpha$:
\begin{enumerate}
\item if $\beta < \alpha$ and $\sigma\in \Gamma_\beta$,
then  $a^\beta_\sigma = a^\alpha_\sigma$,
\item    $\dom(\sigma) <\xi_\sigma$, $e_{\dom(\sigma)}\in M_{\xi_\sigma}$
 and 
   $\{ a^\alpha_{\sigma\restriction \beta} : \beta\in\dom(\sigma)\}\in
 M_{\xi_\sigma}$, 
\item if $\sigma^\frown 0\in \Gamma_\alpha$, then
       $\{ n : c_\sigma(n) \subset a^\alpha_{\sigma^\frown 0}\}$ splits every
infinite $b\subset\omega$ in $M_{\xi_\sigma}$ where,
for  each  $n\in\omega$,  $c_\sigma(n)$ denotes
the set $  a^\alpha_{\sigma\restriction e_{\dom(\sigma)}(n)}\setminus
     \bigcup\{  a^\alpha_{\sigma\restriction e_{\dom(\sigma)}(m)} :
     m<n\}$.
\end{enumerate}
The inductive construction is routine and can be omitted. We note
that properties (2) and (3) ensure that each of
$\{ a^\alpha_{\sigma\restriction \beta} : \beta \in \dom(\sigma)\}\cup
\{ a^\alpha_{\sigma^\frown 0}\}$ 
and $\{ a^\alpha_{\sigma\restriction \beta} : \beta \in \dom(\sigma)\}\cup
\{ a^\alpha_{\sigma^\frown 1}\}$  are proper coherent sequences. We
finish by proving that $X(\Gamma_{\omega_1})$ has no converging
sequences. Let $\{ x_n : n\in \omega\}$ be an infinite subset of
$2^{\omega_1}$ and we show that the sequence does not converge to
$x\in 2^{\omega_1}$. Let $\varphi_x$ be the mapping as in Lemma
\ref{sixteen}, and let, for $n\in\omega$, $y_n = \varphi_x(x_n)$.
Following Definition \ref{fifteen}, let
 $a^x_\alpha = a_{x\restriction\alpha{+}1}$ for all $\alpha\in
\omega_1$. 
It suffices to find a  $\beta  < \omega_1$ so that
$\{ n : y_n\in \widehat{a^x_{\beta}}\}$ is infinite. If
$\{ y_n : n\in\omega\}$ is finite, then this is immediate, so assume
that it is infinite. For each $k\in \omega$, let $\beta_k{+}1 =
e_{\dom(\sigma)}(k)$, and
for each $n\in\omega$, choose the minimal
$k_n\in\omega$ so that $y_n\in \widehat{a^x_{\beta_{k_n}}}$.

Now let $L$ be the infinite set
$ \{ x\restriction y_n{+}1 : n\in \omega\}$ and choose $\sigma
\in \{ x\restriction \beta : \beta<\omega_1\}$ large enough so
that $\{y_n :n\in\omega\}\subset\dom(\sigma)$ and each
of $L$ and $\{ k_n : n\in\omega\}$
are elements of $ M_{\xi_\sigma}$.
Let $\dom(\sigma) = \alpha$
and let
$\{ c_\sigma(k): k\in\omega\}$ be as in condition (3).
Then there is an infinite set $b_\sigma$ chosen
so that $a^{\alpha{+}1}_{\sigma^\frown 0} =
 \bigcup \{ c_\sigma(k) : k \in b_\sigma\}$. 
Since $a^{\alpha{+}1}_{\sigma^\frown 1} = \omega\setminus
a^{\alpha{+}1}_{\sigma^\frown 0}$, we also have that
$a^{\alpha{+}1}_{\sigma^\frown 1} =
\bigcup \{ c_\sigma(k) : k \in \omega\setminus b_\sigma\}$.
Since $a^x_\alpha$ is one of
$a^{\alpha{+}1}_{\sigma^\frown 0},
a^{\alpha{+}1}_{\sigma^\frown 1}$, we have that
$b^x_\alpha = \{ k\in \omega : c_\sigma(k) \subset a^x_\alpha\}$
splits
$\{ k_n : n\in \omega\}$. We finish by checking that
$y_n \in \widehat{a^{\alpha{+}1}_{\alpha}}$ for the infinitely many
 $n$ such that $k_n\in b^x_\alpha$. Fix any $n$ with $k_n\in
b^x_\alpha$ and recall that 
 $y_n \in \widehat{a^\alpha_{\beta_{k_n}}}
\setminus \bigcup_{m<k_n}\widehat{a^\alpha_{\beta_m}}$.
Since $c_\sigma(k_n) = 
{a^\alpha_{\beta_{k_n}}}
\setminus \bigcup_{m<n}{a^\alpha_{\beta_m}}$, it follows
that $y_n\in \widehat{a^x_\alpha}$ since
$
\widehat{a^x_\alpha}$ contains
$\widehat{a^\alpha_{\beta_{k_n}}}
\setminus \bigcup_{m<k_n}\widehat{a^\alpha_{\beta_m}}$.
\end{proof}


\begin{bibdiv}

\def\cprime{$'$} 

\begin{biblist}

  		
\bib{Blass1}{article}{
   author={Blass, Andreas},
   title={Selective ultrafilters and homogeneity},
   journal={Ann. Pure Appl. Logic},
   volume={38},
   date={1988},
   number={3},
   pages={215--255},
   issn={0168-0072},
   review={\MR{942525}},
   doi={10.1016/0168-0072(88)90027-9},
}



\bib{BrendleShelah}{article}{
   author={Brendle, J\"{o}rg},
   author={Shelah, Saharon},
   title={Ultrafilters on $\omega$---their ideals and their cardinal
   characteristics},
   journal={Trans. Amer. Math. Soc.},
   volume={351},
   date={1999},
   number={7},
   pages={2643--2674},
   issn={0002-9947},
   review={\MR{1686797}},
   doi={10.1090/S0002-9947-99-02257-6},
}
	
\bib{withWill}{article}{
author={W.~Brian and A. Dow},
title={Small cardinals and small Efimov spaces
},
note = {submitted},
date={2019},
}


		
\bib{SplitCharacter}{article}{
   author={Dow, Alan},
   title={Efimov spaces and the splitting number},
   note={Spring Topology and Dynamical Systems Conference},
   journal={Topology Proc.},
   volume={29},
   date={2005},
   number={1},
   pages={105--113},
   issn={0146-4124},
   review={\MR{2182920}},
}
	
	\bib{Cclosed}{article}{
   author={Dow, A.},
   title={Compact C-closed spaces need not be sequential},
   journal={Acta Math. Hungar.},
   volume={153},
   date={2017},
   number={1},
   pages={1--15},
   issn={0236-5294},
   review={\MR{3713559}},
   doi={10.1007/s10474-017-0739-x},
}

\bib{Roberto}{article}{
   author={Dow, Alan},
   author={Pichardo-Mendoza, Roberto},
   title={Efimov's problem and Boolean algebras},
   journal={Topology Appl.},
   volume={160},
   date={2013},
   number={17},
   pages={2207--2231},
   issn={0166-8641},
   review={\MR{3116583}},
   doi={10.1016/j.topol.2013.09.006},
}

\bib{ShelahEfimov}{article}{
   author={Dow, Alan},
   author={Shelah, Saharon},
   title={An Efimov space from Martin's axiom},
   journal={Houston J. Math.},
   volume={39},
   date={2013},
   number={4},
   pages={1423--1435},
   issn={0362-1588},
   review={\MR{3164725}},
 }

 \bib{ShelahSplitting}{article}{
   author={Dow, Alan},
   author={Shelah, Saharon},
   title={On the cofinality of the splitting number},
   journal={Indag. Math. (N.S.)},
   volume={29},
   date={2018},
   number={1},
   pages={382--395},
   issn={0019-3577},
   review={\MR{3739621}},
   doi={10.1016/j.indag.2017.01.010},
}

\bib{HrusakIsrael}{article}{
   author={Gaspar-Arreola, Miguel \'{A}ngel}, 
   author={Hern\'{a}ndez-Hern\'{a}ndez, Fernando},
   author={Hru\v{s}\'{a}k, Michael},
   title={Scattered spaces from weak diamonds},
   journal={Israel J. Math.},
   volume={225},
   date={2018},
   number={1},
   pages={427--449},
   issn={0021-2172},
   review={\MR{3805653}},
   doi={10.1007/s11856-018-1669-1},
}
		


\bib{JudahShelah}{article}{
   author={Ihoda, Jaime I.},
   author={Shelah, Saharon},
   title={$\Delta^1_2$-sets of reals},
   journal={Ann. Pure Appl. Logic},
   volume={42},
   date={1989},
   number={3},
   pages={207--223},
   issn={0168-0072},
   review={\MR{998607}},
   doi={10.1016/0168-0072(89)90016-X},
}

\bib{Koppelberg}{article}{
   author={Koppelberg, Sabine},
   title={Boolean algebras as unions of chains of subalgebras},
   journal={Algebra Universalis},
   volume={7},
   date={1977},
   number={2},
   pages={195--203},
   issn={0002-5240},
   review={\MR{0434914}},
   doi={10.1007/BF02485429},
}

\bib{Koppelberg2}{article}{
   author={Koppelberg, Sabine},
   title={Minimally generated Boolean algebras},
   journal={Order},
   volume={5},
   date={1989},
   number={4},
   pages={393--406},
   issn={0167-8094},
   review={\MR{1010388}},
   doi={10.1007/BF00353658},
}
		

\bib  {Kosz1}{article}{
    AUTHOR = {Koszmider, Piotr},
     TITLE = {Forcing minimal extensions of {B}oolean algebras},
   JOURNAL = {Trans. Amer. Math. Soc.},
  FJOURNAL = {Transactions of the American Mathematical Society},
    VOLUME = {351},
      YEAR = {1999},
    NUMBER = {8},
     PAGES = {3073--3117},
      ISSN = {0002-9947},
     CODEN = {TAMTAM},
   MRCLASS = {03E35 (03E50 06E15 54A25 54A35 54G99)},
  MRNUMBER = {1467471 (99m:03099)},
MRREVIEWER = {P{\'e}ter Komj{\'a}th},
       DOI = {10.1090/S0002-9947-99-02145-5},
       URL = {http://dx.doi.org/10.1090/S0002-9947-99-02145-5},
}



\bib{Kunen}{book}{
   author={Kunen, Kenneth},
   title={Set theory},
   series={Studies in Logic and the Foundations of Mathematics},
   volume={102},
   note={An introduction to independence proofs},
   publisher={North-Holland Publishing Co., Amsterdam-New York},
   date={1980},
   pages={xvi+313},
   isbn={0-444-85401-0},
   review={\MR{597342}},
}

\bib{KunenTall}{article}{
   author={Kunen, Kenneth},
   author={Tall, Franklin D.},
   title={Between Martin's axiom and Souslin's hypothesis},
   journal={Fund. Math.},
   volume={102},
   date={1979},
   number={3},
   pages={173--181},
   issn={0016-2736},
   review={\MR{532951}},
   doi={10.4064/fm-102-3-173-181},
}

\bib{Laver}{article}{
   author={Laver, Richard},
   title={On the consistency of Borel's conjecture},
   journal={Acta Math.},
   volume={137},
   date={1976},
   number={3-4},
   pages={151--169},
   issn={0001-5962},
   review={\MR{0422027}},
}

\bib{VaughanSS}{collection}{
   author={Vaughan, Jerry},
   title={Open problems in topology. II},
   editor={Pearl, Elliott},
   publisher={Elsevier B. V., Amsterdam},
   date={2007},
   pages={249--256},
   isbn={978-0-444-52208-5},
   isbn={0-444-52208-5},
   review={\MR{2367385}},
}

		

\end{biblist}
\end{bibdiv}
\end{document}